\documentclass[12pt]{article}
\usepackage[utf8]{inputenc}
\usepackage[T1]{fontenc}

\usepackage{caption}
\oddsidemargin \evensidemargin
\usepackage[colorlinks]{hyperref}
\hypersetup{
  linkcolor=[rgb]{0.3,0.3,0.6},
  citecolor=[rgb]{0.2, 0.6, 0.2},
  urlcolor=[rgb]{0.6, 0.2, 0.2}
}

\usepackage[english]{babel}
\usepackage{amsmath}
\usepackage{amssymb}
\usepackage{amsthm}
\usepackage{latexsym}
\usepackage{mathtools}
\usepackage{thmtools}
\usepackage{amsfonts}
\usepackage{mathrsfs}
\usepackage{textcomp}
\usepackage{graphicx}
\usepackage{setspace}
\usepackage{nicefrac}
\usepackage{indentfirst}
\usepackage{wasysym}
\usepackage[normalem]{ulem}
\usepackage{upgreek}
\usepackage{paralist}
\usepackage{xcolor}
\usepackage{etoolbox}
\usepackage{mathdots}
\usepackage{wrapfig}
\usepackage{floatflt}
\usepackage{tensor} 
\usepackage{parskip}
\usepackage{lscape}
\usepackage{stmaryrd}
\usepackage[enableskew]{youngtab}
\usepackage{ytableau}
\usepackage{enumitem}
\usepackage{pifont}
\usepackage{braket}
\usepackage{comment}

\usepackage{tikz}
\usetikzlibrary{cd}
\usetikzlibrary{arrows}
\usetikzlibrary{matrix}
\usetikzlibrary{positioning}
\usetikzlibrary{decorations}
\usetikzlibrary{calc}

\usepackage[all,cmtip]{xy}
\usepackage[labelfont=bf, font={small},margin=2cm]{caption}[2005/07/16]

\usepackage{xcolor}
\definecolor{red}{rgb}{1,0,0}

\newcommand{\vvirg}{ , \dots , }

\newcommand{\calE}{\mathcal{E}}

\newcommand{\calO}{\mathcal{O}}
\newcommand{\calP}{\mathcal{P}}

\newcommand{\calS}{\mathcal{S}}

\newcommand{\calX}{\mathcal{X}}
\newcommand{\calY}{\mathcal{Y}}
\newcommand{\calZ}{\mathcal{Z}}

\newcommand{\bbP}{\mathbb{P}}
\newcommand{\bbQ}{\mathbb{Q}}

\newcommand{\frakS}{\mathfrak{S}}

\newcommand{\frakc}{\mathfrak{c}}

\newcommand{\rmV}{\mathrm{V}}

\renewcommand{\phi}{\varphi}

\newcommand{\dashto}{\dashrightarrow}

\renewcommand{\tilde}[1]{\widetilde{#1}}

\renewcommand{\bar}[1]{\overline{#1}}

\newcommand{\rank}{\mathrm{rank}}

\DeclareMathOperator{\codim}{codim}

\DeclareMathOperator{\Hom}{Hom}

\DeclareMathOperator{\Ann}{Ann}

\newcommand{\Syz}{\mathrm{Syz}}

\newcommand{\Macaulay}{\texttt{Macaulay2} }

\usepackage{bm}

\newcommand{\KK}{\Bbbk}
\newcommand{\PP}{\mathbb{P}}
\newcommand{\ZZ}{\mathbb{Z}}
\newcommand{\QQ}{\mathbb{Q}}
\newcommand{\FF}{\mathbb{F}}

\DeclareMathOperator{\Gr}{Gr}

\DeclareMathOperator{\regH}{reg_H}
\DeclareMathOperator{\regCM}{reg_{CM}}
\newcommand{\Hm}{\mathrm{H}_{\mathfrak m}}
\newcommand{\sing}{\mathrm{sing}}

\DeclareMathOperator{\modend}{end}

\newcommand{\genHF}{\mathrm{genHF}}
\newcommand{\cmap}{\mathfrak{c}}
\newcommand{\leqlex}{\mathrel{\leq_{\textup{lex}}}}
\newcommand{\geqlex}{\mathrel{\geq_{\textup{lex}}}}
\DeclareMathOperator{\Ker}{Ker}

\title{\bf Hilbert Functions of Chopped Ideals}
\author{Fulvio Gesmundo, Leonie Kayser and Simon Telen}
\date{}

\newtheorem{lemma}{Lemma}[section]
\newtheorem{proposition}[lemma]{Proposition}
\newtheorem{corollary}[lemma]{Corollary}
\newtheorem{theorem}[lemma]{Theorem}
\newtheorem{conjecture}{Conjecture}

\theoremstyle{definition}
\newtheorem{definition}[lemma]{Definition}
\newtheorem{remark}[lemma]{Remark}

\newtheorem{examplex}[lemma]{Example}
\newenvironment{example}
  {\pushQED{\qed}\examplex}
  {\popQED\endexamplex}

\usepackage[frozencache]{minted}
\setminted{frame=lines,
rulecolor=\color{white!80!black},
fontsize=\small,
numbers=right,
numbersep=-5pt,
obeytabs=true,
encoding = utf8,
tabsize=4,
escapeinside=||,
breaklines}

\usepackage{csquotes}

\newcommand\blankfootnote[1]{%
  \let\thefootnote\relax\footnotetext{#1}%
  \let\thefootnote\svthefootnote%
}

\begin{document}
\maketitle

\begin{abstract}
\noindent A chopped ideal is obtained from a homogeneous ideal by considering only the generators of a fixed degree. We investigate cases in which the chopped ideal defines the same finite set of points as the original one-dimensional ideal. The  complexity of computing these points from the chopped ideal is governed by the Hilbert function and regularity. We conjecture values for these invariants and prove them in many cases. We show that our conjecture is of practical relevance for symmetric tensor decomposition.
\end{abstract}

\section{Introduction}\label{sec:1}
\blankfootnote{\textit{Keywords.} Hilbert function, Hilbert regularity, syzygy, liaison, tensor decomposition}
\blankfootnote{\textit{2020 Mathematics Subject Classification.} 13D02, 13C40, 14N07, 65Y20}
Let $\KK$ be an algebraically closed field of characteristic $0$, and let $S = \KK[x_0, \ldots, x_n]$ be the polynomial ring in $n+1$ variables with coefficients in $\KK$. With its standard grading, $S$ is the homogeneous coordinate ring of the $n$-dimensional projective space $\PP^n = \PP^n_\KK$. For a tuple of $r$ points $Z = (z_1, \ldots, z_r) \in (\PP^n)^r$, let $I(Z)$ be the associated vanishing ideal, that is
\[ 
I(Z) = \langle \set{f \in S \text{ homogeneous} | f(z_i) = 0 \text{ for }i=1,\dots,r } \rangle_S .
\]
If the set of points $Z \in (\bbP^n)^r$ is general, the Hilbert function $h_{S/I(Z)} \colon t \mapsto {\dim}_\KK (S/I(Z))_t$ is
\begin{equation} \label{eq:genericHF}
    h_{S/I(Z)}(t) \,=\, \min  \{ h_S(t) , r \}
\end{equation} 
 where $ h_S(t) = \binom{n+t}{n}$ is the Hilbert function of the polynomial ring $S$. Here, the word \emph{general} means that \eqref{eq:genericHF} holds for all $Z$ in a dense Zariski open subset of $(\bbP^n)^r$. We also have, for general $Z$, that the ideal $I(Z)$ is generated in degrees $d$ and $d +1$, where $d$ is the smallest integer such that the minimum in \eqref{eq:genericHF} equals $r$ \cite[Thm.~1.69]{IarrKan:PowerSumsBook}.

This work focuses on a modification $I_{\langle d \rangle}$ of the ideal $I(Z)$, called its  \emph{chopped ideal} in degree $d$. This is defined as the ideal generated by the homogeneous component $I(Z)_d$, that is the set of elements of degree $d$ in $I(Z)$. In particular $I_{\langle d \rangle} \subseteq I(Z)$, and strict inclusion holds if and only if $I(Z)$ has generators in degree $d+1$. An elementary dimension count shows that there is a range for $r$ for which this happens, see \eqref{eq:range_r}. 

If $r < h_S(d) - n$ and $Z$ is general, the saturation $(I_{\langle d \rangle})^{\rm{sat}}$ of the chopped ideal with respect to the irrelevant ideal of $S$ coincides with the ideal $I(Z)$. This is proved in \autoref{thm:trisec}. In other words, $I(Z)$ and $I_{\langle d \rangle}$ both cut out $Z$ scheme-theoretically. In particular, $I(Z)$ and $I_{\langle d \rangle}$ coincide in large degrees, and they have the same constant Hilbert polynomial, equal to $r$. 

In cases where $(I_{\langle d \rangle})^{\rm sat} = I(Z)$ and $I_{\langle d \rangle} \neq I(Z)$, there is a range of degrees $d<t<d+e$ with the property that $h_{S/I_{\langle d \rangle}}(t) > h_{S/I(Z)}(t)$. The goal of this work is to determine this \emph{saturation gap}, and understand the geometric and algebraic properties that control it. We illustrate this phenomenon in a first example.

\begin{example}[$n = 2, d = 5$] \label{ex:intro1}
Let $Z$ be a set of $17$ general points in the plane $\bbP^2$. The lowest degree elements of $I(Z)$ are in degree $5$ and, a priori, $I(Z)$ can have minimal generators in degree $5$ and $6$. It turns out that $I(Z)$ is generated by four quintics. In particular, its chopped ideal $I(Z)_{\langle 5 \rangle}$ coincides with $I(Z)$. We provide a simple snippet of code in \Macaulay \cite{M2} to compute $I(Z)$, its chopped ideal, and the corresponding Hilbert functions: in this case, line 4 returns \texttt{true} and the two Hilbert functions coincide. 
\begin{minted}{macaulay2.py:Macaulay2Lexer -x}
loadPackage "Points"
I = randomPoints(2,17);
Ichop = ideal super basis(5,I);
I == Ichop  
for t to 10 list {t, hilbertFunction(t,I), hilbertFunction(t,Ichop)}
\end{minted}
This uses the convenient package \texttt{Points.m2} to calculate the ideal $I(Z)$ \cite{PointsSource}. Now let $Z$ be a set of $18$ general points in $\PP^2$. In this case $I(Z)$ is generated by three quintics \emph{and one sextic}. The chopped ideal $I_{\langle 5 \rangle}$ is the ideal generated by the three quintics. Changing \texttt{17} into \texttt{18} in line 2 provides the code to compute the chopped ideal of $I(Z)$. Now, line 8 returns \texttt{false} and the two Hilbert functions are recorded below:
\[ \begin{array}{c|ccccccccccc}
t & 0 & 1 & 2 & 3 & 4 & 5 & 6 & 7 & 8 & 9 & \dots \\ \hline
h_{S/I(Z)}(t) & 1 & 3 & 6 & 10 & 15 & 18 & 18 & 18 & 18 & 18 & \dots \\
h_{S/I(Z)_{\langle 5\rangle}}(t) & 1 & 3 & 6 & 10 & 15 & 18 & \mathbf{19} & 18 & 18 & 18 & \dots 
\end{array}
\]
We observe that the Hilbert polynomials are the same: they are equal to the constant $18$. However, the Hilbert function of $S/I_{\langle 5 \rangle}$ overshoots the Hilbert polynomial in degree $6$, and then falls back to $18$ in degree $7$. This specific example is explained in detail in \autoref{sec:quintics}. More generally, the goal of this work is to better understand this phenomenon for all $r$.
\end{example}

We will see that the problem of understanding the Hilbert function of the chopped ideal of a set of points is related to several classical conjectures in commutative algebra and algebraic geometry, such as the Ideal Generation Conjecture and the Minimal Resolution Conjecture. Besides this, our motivation comes from computational geometry. In the most general setting, one is given a system of homogeneous polynomials with the task of determining the finite set of solutions $Z$. In a number of applications, the given polynomials generate a subideal of $I(Z)$. Often, this subideal is the chopped ideal $I_{\langle d \rangle}$. This happens, for instance, in classical tensor decomposition algorithms, see \autoref{sec:6}. In order to solve the resulting polynomial system using normal form methods, such as Gr\"obner bases, one constructs a \emph{Macaulay matrix} of size roughly $h_S(d+e)$, where $e$ is a positive integer such that $h_{S/I_{\langle d \rangle}}(d+e) = r$; see \cite{emiris1999matrices,telen2020thesis} for details. Hence, it is important to answer the following question: 
\begin{equation} \label{eq:ourquestion}
    \textit{\enquote{What is the smallest $e_0 > 0$ such that $h_{S/I_{\langle d \rangle}}(d+e_0) = r$?}}  
\end{equation} 
\autoref{ex:intro1} analyzes two cases for $d = 5$ in $\bbP^2$: if $r = 17$ then the answer is $e_0 = 1$, and for $r = 18$ the answer is $e_0 = 2$. Interestingly, this means that finding $18$ points from their vanishing quintics using normal form methods is significantly harder than finding $17$ points. 

The following conjecture predicts the Hilbert function of the chopped ideal $I_{\langle d \rangle}$. 

\begin{conjecture}[Expected Syzygy Conjecture]\label{conj:chopped_hilbert_function_intro}
Let $Z$ be a general set of $r$ points in $\bbP^n$ and let $d$ be the smallest value for which $h_S(d) \geq r$. Then for any $e\geq 0$
\begin{equation} \label{eq:expectedhf}
h_{I_{\langle d \rangle}}(d+e) \,=\, \begin{cases}
 \displaystyle\sum_{k \geq 1} (-1)^{k+1} \cdot h_S(d+e - kd) \cdot \binom{h_S(d)-r}{k} & e< e_0 \\
 h_S(d+e) -r & e\geq e_0
\end{cases}
\end{equation}
where $e_0>0$ is the smallest integer such that the summation is at least $h_S(d+e_0)-r$.
\end{conjecture}
The heuristic motivation for this conjecture is that, generically, the equations of degree $d$ of a set of points are \emph{as independent as possible}. More precisely, their syzygy modules are generated by the Koszul syzygies, as long as the upper bound $h_{I_{\langle d \rangle}} \leq  h_{I(Z)}$ allows for it.

Our contribution is a proof of \autoref{conj:chopped_hilbert_function_intro} for many small values of $n, r$, and in several infinite families of pairs $(n,r)$. 
\begin{theorem} \label{thm:main}
\autoref{conj:chopped_hilbert_function_intro} is true in the following cases:
\begin{itemize}
 \item \autoref{thm: rmax in general}: $r_{d,\max} = h_S(d) - (n+1)$ for all $d$ in all dimensions $n$;
 \item \autoref{thm:main3.3}: $r_{d,\min} = \frac{1}{2}(d+1)^2$ when $d$ is odd, in the case $n=2$;
 \item \autoref{lem:interesting_range}: $r \leq \frac{1}{n}\big((n+1)h_S(d) - h_S(d+1)\big)$ and $n \leq 4$ and more generally whenever the Ideal Generation Conjecture holds;
 \item \autoref{thm:computeralgebra}: In a large number of individual cases in low dimension:
 \[
\begin{array}{c|ccccccccc}
n  & 2 & 3  & 4 & 5 & 6 & 7 & 8 & 9 & 10 \\ \hline
r  &  \leq 2343 & \leq 2296 & \leq 1815 & \leq 1272 & \leq 908 & \leq 767 & \leq 479 & \leq 207 & \leq 267 
\end{array}
 \]
\end{itemize}
\end{theorem}
We discuss the role of the Ideal Generation Conjecture mentioned in \autoref{thm:main} in \autoref{sec:2}. We propose a second conjecture, implied by \autoref{conj:chopped_hilbert_function_intro}, which pertains to question \eqref{eq:ourquestion}. For $Z \in (\mathbb{P}^n)^r$, let $I_{\langle d \rangle} = \langle I(Z)_d\rangle_S$ and define $\gamma_n(d,Z) \coloneqq \min \set{ e \in \mathbb{Z}_{>0} | h_{S/I_{\langle d \rangle}}(d+e) = r}$.
\begin{conjecture}[Saturation Gap Conjecture] \label{conj:main}
Let $n,d,r\in\ZZ_{>0}$ be integers with $r < h_S(d)-n$. For general $Z$, the value $\gamma_n(d,Z) = \gamma_n(d,r)$ only depends on $n,d,r$ and it is given explicitly~by 
\[
\gamma_n(d,r) \,=\, \min \Set{ e \in \mathbb{Z}_{>0} | h_S(d+e) - r \leq \sum_{k = 1}^{n-3} (-1)^{k+1}  h_S(d+e - kd) \binom{h_S(d) - r}{k} }.
\]
\end{conjecture}
The fact that the \emph{general} value $\gamma_n(d,r)$ in \autoref{conj:main} only depends on $n,d,r$ is a consequence of a standard semicontinuity argument, see \autoref{prop: semicontinuity multiplication map}. We call this quantity the \emph{saturation gap}. It measures the gap between degrees at which the chopped ideal $I_{\langle d \rangle}$ agrees with its saturation $(I_{\langle d \rangle})^{\rm sat} = I(Z)$. \autoref{thm:main} guarantees that \autoref{conj:main} holds in all listed cases. Moreover, \autoref{cor:boundgapgeneral} provides an upper bound for $\gamma_{n}(d,r)$ for every $n,d,r$. 

The paper is organized as follows. \autoref{sec:2} sets the stage for the study of chopped ideals. It proves some preliminary results and explains the relations to other classical conjectures in commutative algebra and algebraic geometry. \autoref{sec:3} is devoted to the case of points in the projective plane. It includes a detailed explanation of the case of $18$ points in $\bbP^2$, the first non-trivial case, and two results solving \autoref{conj:chopped_hilbert_function_intro} in extremal cases. \autoref{sec: bound rmax general} concerns the proof of \autoref{conj:chopped_hilbert_function_intro} for the largest possible number of points $r = h_S(d) - (n+1)$ for given $n,d$. Moreover, we provide an upper bound for the saturation gap for any number of points. \autoref{sec:5} contains a computational proof for the remaining cases in \autoref{thm:main}. Finally, \autoref{sec:6} discusses symmetric tensor decomposition and its relation to \autoref{conj:main}. The computations in the final two sections use \texttt{Macaulay2} \cite{M2} and \texttt{Julia}; the code to replicate the computations is available online at \url{https://mathrepo.mis.mpg.de/ChoppedIdeals/}.

\section{Chopped ideals} \label{sec:2}

\begin{definition}[Chopped ideal]
Let $I\subseteq S$ be a homogeneous ideal and $d\geq 0$. The \emph{chopped ideal} in degree $d$ associated to $I$ is $I_{\langle d \rangle} \coloneqq \langle I_d \rangle_S$.
\end{definition}

The notation $I_{\langle d \rangle}$ goes back to \cite{HerzogHibi1999Componentwise}, where the authors study ideals whose chopped ideals have linear resolutions.

\subsection{The chopping map}
Denote by $U_{\genHF}\subseteq (\PP^n)^r$ the dense Zariski open subset of $(\PP^n)^r$ consisting of $r$-tuples satisfying \eqref{eq:genericHF}. We focus on the chopped ideals $I_{\langle d \rangle}$ where $I = I(Z)$ for some $Z \in U_{\genHF}$. Moreover, we are interested in the case where $Z$ can be computed from its chopped ideal. To this end, we determine the values of $r$ for which $I(Z)$ and $I_{\langle d \rangle}$ define the same subscheme of $\PP^n$. Given a set of homogeneous polynomials $J \subseteq S$, let $\rmV(J) \subseteq \bbP^n$ denote the subscheme of $\bbP^n$ that they define.

\begin{theorem} \label{thm:trisec}
Let $d\geq 1$, let $Z$ be a general set of $r \leq h_S(d)$ points in $\bbP^n$ and $I_{\langle d \rangle } \coloneqq I(Z)_{\langle d \rangle}$. 
\begin{enumerate}[label=\textup{(\roman*)}]
	\item If $r = h_S(d) - n$, then $\rmV(I_{\langle d \rangle})$ is a set of $d^n$ reduced points.
	\item If $r \geq h_S(d)-n$, then $\rmV(I_{\langle d \rangle})$ is a complete intersection of codimension $h_S(d)-r$.
	\item If $r<h_S(d)-n$, then $\rmV(I_{\langle d \rangle})$ is the reduced scheme $Z$.
\end{enumerate}
\end{theorem}
For the proof, we consider a geometric interpretation of the operation of \emph{chopping} an ideal.
\begin{definition}[Chopping map $\cmap$]\label{def:chopping_map}
For given $d$ with $h_S(d)\geq r$, the \emph{chopping map} is
\[
\cmap \colon U_{\text{genHF}} \to \Gr(h_S(d) - r, S_d), \qquad Z \mapsto [I(Z)_d].
\]
\end{definition}
The chopping map is a morphism of varieties. In fact, there is a commutative diagram involving the Veronese embedding $\nu_d\colon \PP^n \hookrightarrow \PP(S_d^\vee)$
\[\begin{tikzcd}
	{\makebox[0pt][r]{$\PP(S_d^\vee)^r \supseteq$\;}(V_{d,n})^r} & {\Gr(r,S_d^\vee)} \\
	\makebox[0pt][r]{$(\PP^n)^r \supseteq$\;} U_{\text{genHF}} & {\Gr(h_S(d)-r,S_d).}
	\arrow["\cong", from=1-2, to=2-2]
	\arrow["\rm span", dashed, from=1-1, to=1-2]
	\arrow["{(\nu_d)^{\times r}}", hook, from=2-1, to=1-1]
	\arrow["\cmap", from=2-1, to=2-2]
\end{tikzcd}\]
For a linear space $T \subseteq S_d$, the scheme $\rmV(T) \subseteq \PP^n$ is the intersection $\nu_d(\bbP^n) \cap \PP(T^\perp) \subseteq \PP(S_d^\vee)$ under the identification induced by the Veronese embedding, see, e.g. \cite[Prop.~4.4.1.1]{Lan:TensorBook}. Notice that $\frakc$ is invariant under permutation of the factors of $(\bbP^{n})^{r}$, therefore it induces a map on the quotient $\tilde{\frakc}\colon U_{\genHF} / \frakS_{r} \to  \Gr(h_S(d)-r,S_d)$.

\begin{theorem}[Geometry of the chopping map]\label{thm:chop_map}
Let $r,n,d>0$ be integers with $r \leq h_S(d)$.
\begin{enumerate}[label=\textup{(\roman*)}]
\item If $r \geq  h_S(d)-n$, then $\cmap$ is dominant, with general fiber of dimension $nr - r(h_S(d)-r)$.
\item If $r \leq h_S(d)-n$, then $\cmap$ is generically finite. More precisely, the induced map $\tilde{\frakc}$ has degree $\binom{d^n}{r}$ if $r = h_S(d)-n$, it is generically injective otherwise.
\end{enumerate}
\end{theorem}

\begin{proof}[Proof of \autoref{thm:trisec} and \autoref{thm:chop_map}]
First consider the case $r = h_S(d)-n = 1+\codim \nu_d(\bbP^n)$. A general linear space $\Lambda \in \Gr(r,S_d^\vee)$ intersects  $\nu_d(\bbP^n)$ in a non-degenerate set of reduced points \cite[Prop.~18.10]{Harris1992}. Picking $r$ points on $\nu_d(\bbP^n)$ spanning $\Lambda$, we see that the map $\cmap$ is dominant. Furthermore, $\Lambda \cap \nu_d(\bbP^n)$ consists of $\deg V_{d,n} = d^n$ reduced points. By genericity, any subset of $r$ points span $\Lambda$, hence $\tilde{\cmap}^{-1}(\Lambda^\perp)$ consists of $\binom{d^n}{r}$ points in $U_{\textup{genHF}} / \mathfrak{S}_r$.

Next, let $r > r'$ where $r' = h_S(d)-n$, and let $U' \subseteq (\PP^n)^{r'}$ be the open set from the previous case. For any set of $r$ points $Z\in U_{\textup{genHF}}$ and containing a subset $Z'$ belonging to $U'$, we must have $\dim {\rm V}(I(Z)_d) = \#(Z\setminus Z')$. Indeed, modulo $I(Z)_d$, there are $\#(Z\setminus Z')$ linearly independent elements in $I(Z')_d$: if $\dim {\rm V}(I(Z)_d) > \#(Z\setminus Z')$ then these additional equations could not cut out the $0$-dimensional set of points $Z'$, in contradiction with the previous part of the proof. Since $S$ is graded Cohen-Macaulay, this implies that a basis of $I(Z)_d$ is a regular sequence \cite[Thm.~17.4]{Matsu:CommutativeRingTheory}. This shows that for any generic enough $Z$, the variety ${\rm V}(I(Z)_d)$ is a complete intersection of dimension $r-(h_S(d)-n)$. Proving that the chopping map is dominant is done exactly as in the previous case, the claim about the fiber dimension is a dimension~count.

Finally, consider $r < h_S(d)-n$. We give a proof valid in characteristic $0$. The classical \emph{Multisecant Lemma} \cite[Prop.~1.4.3]{Russo2016:ProjGeo} states that a general $k$-secant plane to a non-degenerate projective variety $X \subseteq \PP^N$ is not a $k+1$-secant for $k< \codim X$. Applying this to the Veronese variety, for a general set of $r < h_S(d)-n$ points $Z$, we have
\[
\nu_d({\rmV}(I(Z)_d)) \,=\,  \nu_d(\bbP^n)  \cap \PP(I(Z)_d^\perp) \,=\,  \nu_d(\bbP^n) \cap \langle \nu_d(Z) \rangle_\PP  \,=\,  \nu_d(Z).
\]
This shows that $\rmV(I(Z)_d) = Z$ for general $Z$. In particular, $\Tilde{\cmap}$ is generically injective.
\end{proof}
This answers our question on when $Z$ can be recovered from the chopped ideal.
\begin{corollary}\label{thm:chopCutsZ}
With notation as before, for general $Z$ the following are equivalent:
\begin{enumerate}[label=\textup{(\roman*)}]
\item ${\rm V}(I(Z)_d) = Z$, as reduced schemes;
\item $(I_{\langle d \rangle})^{\rm sat} = I(Z)$, where $\cdot^{\rm sat}$ denotes saturation with respect to $\langle x_0, \dots, x_n \rangle_S$;
\item $r < h_S(d)-n$ or $r=1$ or $(n,r) = (2,4)$.
\end{enumerate}
\end{corollary}
\begin{proof}
By the projective Nullstellensatz, we have $J^{\rm sat} = I({\rm V}(J))$, this shows the equivalence of (i) and (ii). If $r < h_S(d)-n$, then by \autoref{thm:trisec} ${\rm V}(I(Z)_{\langle d \rangle}) = Z$. If $r \geq h_S(d) - n$, then ${\rm V}(I(Z)_{\langle d \rangle})$ is a complete intersection, which, for general $Z$, only happens if $Z$ is a single point or four points in $\PP^2$.
\end{proof}

Inspecting the proof of \autoref{thm:trisec} and \ref{thm:chop_map}, we make a useful technical observation.

\begin{remark}\label{lem:IZ regular sequences}
For  $n,d,r$ such that $h_S(d)-n \leq r < h_S(d)$ and $Z$ general, a general collection of polynomials $f_1,\dots,f_{s} \in I(Z)_d$ is a regular sequence when $s\leq h_S(d)-r$. If $s=n=h_S(d) - r$, then this complete intersection is reduced according to \autoref{thm:trisec}(i).
\end{remark}

\subsection{The saturation gap and expected syzygies}

From now on, for fixed $n,r$, set $d \coloneqq \min\set{t | h_S(t) \geq r}$. Let $Z$ be a set of $r$ general points in $\PP^n$, with vanishing ideal $I=I(Z)$. The degree $d$ is the \emph{Hilbert regularity} of $Z$ (or $S/I(Z)$) defined for a finite graded $S$-module $M$ by
\[
\regH(M) \,\coloneqq\, \min \Set{d \in \ZZ | h_M(t) = \operatorname{HP}_M(t) \text{ for }t\geq d}.
\]
The minimal generators of $I=I(Z)$ are in degrees $\{d,d+1\}$ \cite[Thm.~1.69]{IarrKan:PowerSumsBook} and the operation of chopping the ideal in degree $d$ is trivial unless $I$ has generators in degree $d+1$. The number of minimal generators in degree $d$ is $h_I(d) = h_S(d) -r$ by assumption \eqref{eq:genericHF}, while the minimal generators in degree $d+1$ span a complement of $S_1I_d$ in $I_{d+1}$.
The linear space $S_1I_d$ is the image of the multiplication map $\mu_1\colon S_1\otimes_\KK I_d \to I_{d+1}$; the expected dimension of this image is
\[
\min\{h_S(1)\cdot h_I(d), h_I(d+1)\} \,=\, \min\{(n+1)\cdot(h_S(d)-r), h_S(d+1)-r\}.
\]
which is always an upper bound and is achieved if and only if $\mu_1$ has maximal rank. This leads to the following long standing conjecture \cite{geramita1984ideal}.
\begin{conjecture}[Ideal generation conjecture] \label{conj:igc}
    Let $n,r,d \geq 1$ be integers such that $r \leq h_S(d)$. There is a Zariski dense open subset $U_{\rm igc} \subseteq (\PP^n)^r$ such that, for $Z \in U_{\rm igc}$, the number of minimal generators of $I(Z)$ in degree $d+1$ is $\max \{ 0, h_{S}(d+1) - r - (n+1) \cdot (h_S(d) - r) \}$. 
\end{conjecture}

From this we see that $I(Z)$ has generators in degree $d+1$ if $h_S(d+1)-r - (h_S(d)-r)(n+1)>0$, or equivalently 
\begin{equation*} \label{eq:lowerboundr}
    r \,>\, \frac{(n+1)h_S(d) - h_S(d+1)}{n}.
\end{equation*} 
This bound is sharp assuming the ideal generation conjecture, which is known to hold for $n \leq 4$ or $r$ large, see \autoref{sec: conjectures}. In fact, using \autoref{thm:chopCutsZ}, we can pinpoint the range in which the chopped ideal cuts out $Z$ in a non-saturated way.
\begin{lemma}\label{lem:interesting_range}
Let $n,d$ be positive integers. If 
\begin{equation} \label{eq:range_r}
   \frac{(n+1)h_S(d) - h_S(d+1)}{n} \,<\, r \,<\, h_S(d)- n,
\end{equation}
then a general set of $r$ points in $\PP^n$ has Hilbert regularity $d$, $\rmV(I(Z)_d) = Z$ but $I_{\langle d \rangle} \subsetneq I$. If the Ideal Generation Conjecture holds for $n,d$, then the lower bound is tight.
\end{lemma}

\begin{remark}
Note that $\frac{(n+1)h_S(d) - h_S(d+1)}{n} \geq h_S(d-1)$. In particular, in the interesting range, equations of degree $d$ are equations of minimal degree.
\end{remark}

In light of the Ideal Generation Conjecture, our \autoref{conj:chopped_hilbert_function_intro} is a natural generalization; it claims that the multiplication map
$
\mu_e \colon I_d \otimes_\KK S_e \to I_{d+e}
$
has the largest possible rank. To give a formal upper bound on the rank of $\mu_e$, we introduce the \emph{lexicographic ordering} on functions $h,h' \colon \ZZ_{\geq 0} \to \ZZ$:
\[
h \leqlex h' \quad \text{if and only if} \quad \inf \set{t| h(t) < h'(t)} \,\leq\, \inf \set{t| h(t) > h'(t)}.
\]
In other words, $h <_{\text{lex}} h'$ if $h(t) = h'(t)$ for $t<t_0$ and $h(t_0) < h'(t_0)$. This is a total order on functions $h$, and a pointwise inequality $h(t) \leq h'(t)$ for all $t \geq 0$ implies $h \leqlex h'$.

An important theorem of Fröberg \cite{Fro:IneqHilbertSer} asserts the following lower bound.
\begin{theorem}\label{thm:fröberg}
For any ideal $J \subseteq S$ of depth $0$ generated by $s\geq n+1$ elements in degree $d$ one has
\[
h_{S/J}(t) \,\geqlex\, \textup{frö}_{d,s}(t) \,\coloneqq\, \begin{cases}
\sum_{k\geq 0} (-1)^k h_S(t-kd) \binom{s}{k} & \text{if } t<t_0 \\
0 & \text{if } t\geq t_0,
\end{cases}
\]
where $t_0\geq 0$ is the first value for which the summation becomes nonpositive.
\end{theorem}
The \emph{Fröberg Conjecture} predicts that equality is achieved for general $J$. If the conjecture holds (for particular $n,d,s$) the lex-inequality is upgraded to a pointwise inequality for \emph{all} such ideals $J$:
\[
h_{S/J}(t) \,\geq\, \text{frö}_{d,s}(t) \quad \text{for all }t\geq 0.
\]
In our situation this leads to the following theorem.
\begin{theorem}\label{thm: ESC lower bound}
If $I \subseteq S$ has dimension $1$, degree $r<h_S(d)-n$ and Hilbert function \eqref{eq:genericHF},~then
\[
h_{S/I_{\langle d \rangle}}(t) \,\geqlex\, \begin{cases}
\textup{frö}_{d,h_I(d)}(t) & \text{if } t<t_1 \\
r & \text{if } t\geq  t_1,
\end{cases}
\]
where $t_1 = \inf \set{t>d| \textup{frö}_{d,h_I(d)}(t)\leq r}$. More precisely, if $h_{S/I_{\langle d \rangle}}(t')\leq r$ for some $t'>d$, then $h_{S/I_{\langle d \rangle}}(t) = r$ for $t\geq t'$. 
\end{theorem}
\begin{proof}
By \cite[Lem.~1]{Fro:IneqHilbertSer}, $h_{S/I_{\langle d \rangle}}(t) \geq h_{S/J}(t)$ where $J$ is generated by $h_I(d)$ general forms of degree $d$. Applying Fröberg's \autoref{thm:fröberg} from above to $J$, which has dimension and depth $0$, we obtain $h_{S/I_{\langle d \rangle}} \geqlex \text{frö}_{d,h_I(d)}$.
Furthermore, if $h_{S/I_{\langle d \rangle}}(t')\leq r$ for some $t'>d$, then $(I_{\langle d \rangle})_{t'} = I_{t'}$. Since the minimal generators of $I(Z)$ are in degree at most $d+1\leq t'$, we have $(I_{\langle d \rangle})_{t} = I_{t}$ for $t\geq t'$ and hence $h_{S/I_{\langle d \rangle}}$ sticks to $r$ from that point on.
\end{proof}

\autoref{conj:chopped_hilbert_function_intro} predicts that for $Z$ general, the Hilbert function $h_{S/I(Z)_{\langle d \rangle}}$ satisfies \autoref{thm: ESC lower bound} with equality, which then is upgraded to a pointwise lower bound. In particular, the multiplication map $\mu_e: I_d \otimes_\KK S_e \rightarrow I_{d+e}$ is either surjective onto $I_{d+e}$, or it achieves the maximum possible dimension from \autoref{thm: ESC lower bound}:
\[
h_{I_{\langle d \rangle}}(d+e) \,=\, \sum_{k \geq 1} (-1)^{k+1} \cdot h_S(d+e - kd) \cdot \binom{h_S(d)-r}{k}
\]
until this sum falls below $h_{I}(d+e)$, from which point on $h_{I_{\langle d \rangle}}(d+e) = h_S(d+e) - r$.

\subsection{Related open problems in commutative algebra} \label{sec: conjectures}

In this section, we give an overview of several conjectures in the study of ideals of points, related to \autoref{conj:chopped_hilbert_function_intro} and \autoref{conj:main}. Let $Z$ be a set of $r$ general points in $\bbP^n$, let $I = I(Z)$ be the vanishing ideal and let $d$ be the Hilbert regularity of $Z$. 

The multiplication map $I_{t}\otimes S_e \to I_{t+e}$ is surjective for $t \geq d+1$ because all generators of $I(Z)$ are in degrees $d$ and $d+1$. The already mentioned \emph{Ideal Generation Conjecture} (IGC) stated in \autoref{conj:igc}, predicts that $\mu_1\colon I_d \otimes S_1 \to I_{d+1}$ has full rank: in other words, either $\mu_1$ is surjective or $I$ has exactly $h_S(d+1) - (n+1)(h_S(d)-r)$ generators of degree $d+1$. This is  related to \autoref{conj:chopped_hilbert_function_intro}, which predicts that $\mu_e \colon I_d \otimes S_e \to S_{d+e}$ has the \emph{expected rank}, and its kernel arises from the Koszul syzygies of the degree $d$ generators of $I$.

The \emph{Minimal Resolution Conjecture} (MRC) \cite{lorenzini1993minimal} is a generalization of the IGC which predicts the entire \emph{Betti table} of the ideal $I$. Consider the minimal free resolution of  $S/I$, regarded as an $S$-module:
\[
\begin{tikzcd}[sep=scriptsize]
	0 & S/I \arrow[l] & S \arrow[l] & F_1 \arrow[l] & \dotsm \arrow[l] & F_{{\rm pd}(S/I)} \arrow[l] & 0, \arrow[l]
\end{tikzcd} \qquad F_i = \bigoplus_j S[-j]^{\beta_{i,j}}.
\]
A consequence of \cite[Lem.~1.69]{IarrKan:PowerSumsBook} is that, for $i \geq 1$, there are at most two nonzero Betti numbers; they are $\beta_{i,d+i-1}$ and $\beta_{i,d+i}$. The IGC predicts that either $\beta_{1,d+1} = 0$ or $\beta_{2,d+1} = 0$. The MRC predicts all values $\beta_{i,j}$. 

Notice that if $\beta_{2,d+1} = 0$, then $\beta_{1+i,d+i} = 0$ for every $i \geq 1$ as well; in this case the values $\beta_{1,d},\beta_{1,d+1}$, together with the exactness of the resolution, uniquely determine the other $\beta_{i,j}$. This is expected to be the case in the range of \eqref{eq:range_r}. In particular, in this range the MRC and the IGC are equivalent and, in a sense, \autoref{conj:chopped_hilbert_function_intro} is a generalization of both.

The MRC is known to be true for $n = 2$ \cite{geramita1984ideal,GerGreRob:MonomialIdeals}, for $n=3$ \cite{ballico1987generators} and for $n=4$ \cite{Walt:MinimalFreeResolutionP4}. Moreover, it has been proved in an asymptotic setting \cite{HirSimp:ResolutionMinimaleGrandNombre} and in a number of other sporadic cases, for which we refer to \cite{Eisenbud1996GaleDA}. It is however false in general \cite{EPSW:ExteriorAlgebraMethods}. There are no known counterexamples to the IGC.

We record here the statement in the case of $\bbP^2$, where the Hilbert-Burch theorem dictates the structure of the minimal free resolution of $S/I$ \cite[Thm.~3.2, Prop.~3.8]{Eisenbud2005:SyzygyBook}. 
\begin{proposition}[Minimal resolution conjecture in $\PP^2$] \label{prop: MRC in P2}
For a general collection of $r$ points $Z\subseteq \PP^2$ with $\regH(Z) = d$, the minimal free resolution of $S/I(Z)$ has the form
\[
\begin{tikzcd}[sep=scriptsize]
	0 & S/I(Z) \arrow[l] & S \arrow[l] & {\!\begin{array}{c}S[-d]^{\beta_{1,d}}\\\bigoplus\\S[-(d+1)]^{\beta_{1,d+1}}\end{array}\!} \arrow[l] & {\!\begin{array}{c}S[-(d+1)]^{\beta_{2,d+1}}\\\bigoplus\\S[-(d+2)]^{\beta_{2,d+2}}\end{array}\!} \arrow[l, "B"'] & 0. \arrow[l]
\end{tikzcd}
\]
Here $\beta_{1,d} = h_S(d) - r$, $\beta_{1,d+1} = \max\{0, h_S(d+1) - (n+1) \beta_{1,d} - r\}$, $\beta_{2,d+1}+\beta_{2,d+2} = \beta_{1,d} + \beta_{1,d+1}-1$ and $\beta_{1,d+1}\cdot \beta_{2,d+1} = 0$.
\end{proposition}
For a proof of this particular case, see for example \cite[Prop.~1.7]{Sauer1985:PointsP2}. In the paper the proof goes via polarization of monomial ideals. One might expect a similar approach would yield \autoref{conj:chopped_hilbert_function_intro} in $\PP^2$. This is not the case, as we will show in \autoref{thm:nomonomials}.

The already mentioned Fröberg's Conjecture \cite{Fro:IneqHilbertSer} predicts the Hilbert function of the ideal generated by generic forms. \autoref{conj:chopped_hilbert_function_intro} states that chopped ideals of general points satisfy Fröberg's conjecture for as many values of $t\in \ZZ_{\geq 0}$ as they \emph{possibly} can.

\subsection{Castelnuovo-Mumford Regularity}
We discussed relations of \autoref{conj:chopped_hilbert_function_intro} with the IGC and the MRC, which have a more cohomological flavour. This raises questions about other cohomological invariants of chopped ideals. We prove a statement regarding the \emph{Castelnuovo-Mumford regularity} of $I_{\langle d \rangle}$. For a finite graded $S$-module $M$, this is defined  as $\regCM M = \max\set{ j - i | \beta_{i,j}(M) \neq 0}$, where $\beta_{i,j}$ are the graded Betti numbers of $M$.
\begin{theorem}\label{thm:regCM}
Let $J \subseteq S$ be a one-dimensional graded ideal, then
\begin{equation}\label{eq:regularities}
\regCM S/J \,=\, \max \{\regH S/J-1, \regH S/J^{\rm sat} \}.
\end{equation}
\end{theorem}
Applying this theorem to a chopped ideal of a general set of points, we obtain:
\begin{corollary}
Let $n,r,d$ satisfy \eqref{eq:range_r}. Then for a general set of $r$ points
\[
\regCM S/I_{\langle d \rangle} \,=\, \regH S/I_{\langle d \rangle} - 1.
\]
\end{corollary}
The \autoref{conj:main} predicts the Hilbert regularity of $I_{\langle d \rangle}$, hence this conjecture is directly related to Castelnuovo-Mumford regularity.

\begin{proof}[Proof of \autoref{thm:regCM}]
The proof relies on local cohomology. Let ${\mathfrak m} = \langle x_0, \ldots, x_n \rangle_S$. The 0-th local cohomology group measures non-saturatedness as
\[
\Hm^0(S/J) \,=\, \Set{x + J \in S/J | \mathfrak m^k x \subseteq J,\ k\gg 0} \,=\, J^{\rm sat}/J.
\]
The dimension of a finite $S$-module $M$ can be characterized as the largest $i\geq 0$ with $\Hm^i(M) \neq 0$ \cite[Prop.~A1.16]{Eisenbud2005:SyzygyBook} so all cohomology groups $\Hm^i(S/J)$ vanish for $i\geq 2$.

We next provide a description of the remaining first local cohomology group. Let $I = J^{\rm sat}$ and let $Z = \operatorname{Proj} S/I \subseteq \PP^n$ be the scheme defined by $I$. The quotients $S/I$ and $S/J$ have the same higher local cohomology. The comparison sequence for local and sheaf cohomology
\[
\begin{tikzcd}[sep=scriptsize]
0 \arrow[r] &  S/I \arrow[r] & \displaystyle\bigoplus_{d \in \ZZ}\underbrace{{\rm H}^0(\PP^n, \mathcal{O}_{Z}(d))}_{\cong\, \KK^{\deg Z}} \arrow[r] & \Hm^1(S/I) \arrow[r] & 0.
\end{tikzcd}
\]
shows that $\Hm^1(S/I)_d$ is the cokernel of $(S/I)_d \hookrightarrow {\rm H}^0(\PP^n, \mathcal{O}_{Z}(d))$. Introducing the notation $\modend(N) \coloneqq \sup \set{t \in \ZZ | N_t \neq 0}$,  this shows that $\modend (\Hm^1(S/J)) + 1 = \regH S/I \eqqcolon d$.

Now the Castelnuovo-Mumford regularity can be expressed in terms of local cohomology: 
\[
\regCM M \,=\, \max_i \modend(\Hm^i(M)) + i.
\]
See \cite[Thm.~4.3]{Eisenbud2005:SyzygyBook}. For $M=S/J$ using vanishing in degree $i \geq \dim S/J = 1$, this gives
\begin{equation}\label{eq:CM_for_SJ_max}
\regCM S/J \,=\, \max \{\modend(I/J), d\}.
\end{equation}
To relate this to the maximum in equation \eqref{eq:regularities}, we distinguish two cases. If $\modend(I/J) \geq d$, then $h_{S/J}(t) > h_{S/I}(t)$ for some $t\geq d$, so $\modend(I/J) = \regH S/J - 1$. Otherwise $\modend(I/J)+1 \leq d$, then $\regH S/J \leq \regH S/I$ and the maximum in  \eqref{eq:CM_for_SJ_max} is attained at $d$.\qedhere
\end{proof}

\section{Points in the plane} \label{sec:3}

When $n =1$, $Z$ is a set of points on the projective line. In this case $I(Z)$ is a principal ideal and it always coincides with its chopped ideal. In particular, \autoref{conj:chopped_hilbert_function_intro} and \autoref{conj:main} hold trivially, as well as the IGC and the MRC.

This section studies the case $n = 2$, that is when $Z \in (\mathbb{P}^2)^r$ is a collection of $r$ general points in the plane. \autoref{fig:gapsP2} shows the saturation gaps for some values of $r$. In each case we use the chopped ideal $I_{\langle d \rangle}$ in degree $d = \min \set{ t | h_S(t) \geq r }$. The gap is only plotted in cases where $I_{\langle d \rangle}$ defines $Z$ scheme-theoretically, following \autoref{thm:chopCutsZ}. 
\begin{figure}
    \centering
    \includegraphics[width=0.8\linewidth]{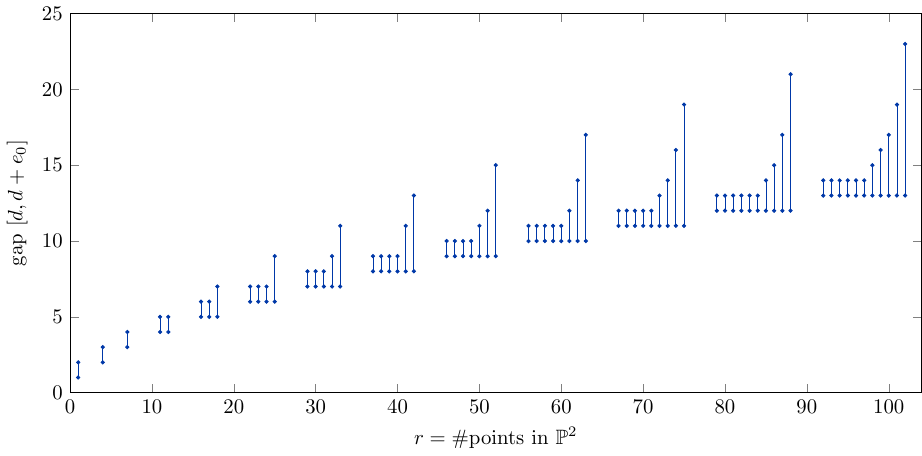}
    \caption{Saturation gaps for chopped ideals of points in the plane.}
    \label{fig:gapsP2}
\end{figure}
Since the IGC is known to be true in $\bbP^2$, \autoref{lem:interesting_range} provides exactly the range where $I_{\langle d \rangle} \neq I(Z)$, and they both define the set $Z$ as a scheme:
\begin{equation} \label{eq:interesting_range_plane}
\frac{d(d+2)}{2}  \,<\,  r  \,<\,  \frac{(d+2)(d+1)}{2}-2 .
\end{equation}
If $d < 5$, this range is empty and the corresponding gaps in  \autoref{fig:gapsP2} have length 1. For $d = 5$, the only integer solution to \eqref{eq:interesting_range_plane} is $r = 18$. This is the leftmost length-two gap in \autoref{fig:gapsP2}. Hence, the simplest interesting chopped ideal is that of \emph{three quintics passing through 18 general points in the plane}, first encountered in \autoref{ex:intro1}. In \autoref{sec:quintics}, we thoroughly work out this instructive example.  
For $d \geq 5$, write 
\[
r_{d,{\min}}  \,\coloneqq\,  \left \lfloor \frac{d(d+2)}{2} \right \rfloor + 1  \qquad \text{and} \qquad r_{d,{\max}} \,=\,  \frac{(d+2)(d+1)}{2}-3  
\]
for the extremal values in the range \eqref{eq:interesting_range_plane}. \autoref{sec:rdmax} proves \autoref{conj:chopped_hilbert_function_intro} for $r = r_{d,\max}$, and \autoref{sec:rdmin} proves it for $r = r_{d,\min}$, when $d$ is odd. Throughout this section, $S = \KK[x_0,x_1,x_2]$.

\subsection{Quintics through eighteen points} \label{sec:quintics}

Let $Z = (z_1, \ldots, z_{18})\in (\mathbb{P}^2)^{18}$ be a configuration of $18$ general points in $\bbP^2$. Equation \eqref{eq:genericHF} guarantees that $Z$ has no equations of degree $4$ and exactly $3 = 21-18$ equations of degree $5$. Hence $I(Z)_5 = \langle f_0,f_1,f_2\rangle_\KK$ for three linearly independent elements $f_i \in S_5$, and the chopped ideal is ${I_{\langle 5 \rangle}}= \langle f_0,f_1 ,f_2 \rangle_S$.

Notice that $h_{S/{I_{\langle 5 \rangle}}}(6) \geq 28 - 3 \cdot 3 = 19$ and equality holds if and only if the three quintics $f_0,f_1,f_2$ do not have linear syzygies. Since the IGC is true for $n=2$, this is indeed the case. Moreover, \autoref{ex:intro1} shows experimentally that $h_{S/{I_{\langle 5 \rangle}}}(7) = 18$, so the saturation gap is $2$. The minimal resolution of $I(Z)$ according to \autoref{prop: MRC in P2} is
\[
\begin{tikzcd}[sep=scriptsize]
	0 & I(Z) \arrow[l] & {\!\begin{array}{c}S[-5]^{3}\\\bigoplus\\S[-6] \end{array}\!} \arrow[l] & S[-7]^{3}\arrow[l, "B"'] & 0. \arrow[l]
\end{tikzcd}
\]
The minimal generators of $I(Z)$ are the maximal minors of the Hilbert-Burch matrix $B$, which is a $4 \times 3$-matrix with three rows of quadrics and one row of linear forms. As a result, the maximal minors are three quintics, spanning the linear space $I(Z)_5$, and one sextic, which is an element of $I(Z)_6 \setminus ({I_{\langle 5 \rangle}})_6$. The existence of the sextic is predicted by \autoref{thm: ESC lower bound} and the gap $\gamma_n(d,r) = \gamma_2(5,18) = 2$ agrees with \autoref{conj:main}.

Note that the \emph{missing sextic} is uniquely determined modulo the $9$-dimensional linear space $({I_{\langle 5 \rangle}})_6 \subseteq I(Z)_6$. We provide a way to compute an element of $I(Z)_6 \setminus ({I_{\langle 5 \rangle}})_6$ from $Z$. 
\begin{proposition}\label{prop: missing sextic}
 Let $Z \subseteq \bbP^2$ be a set of $18$ general points and let $Z = Z_1 \mathbin{\dot{\cup}} Z_2$ be a partition of $Z$ into two sets of $9$ points. Let $f_1,f_2,f_3 \in I(Z)_5 \subseteq S_5$ be linearly independent and let $g_i \in I(Z_i)_3 \setminus \{0\} \subseteq S_3$. Then $g = g_1g_2 \in I(Z)$ and $g \notin I_{\langle 5 \rangle}$.
\end{proposition}
The proof of \autoref{prop: missing sextic} is deferred to \autoref{app: missing sextic}. 

It was observed in \autoref{ex:intro1} that $(I_{\langle 5 \rangle})_7 = I(Z)_7$. This is equivalent to the following result, which is a consequence of the more general \autoref{thm: rmax in P2} and \autoref{thm:main3.3}. 

\begin{proposition} \label{prop:noquadsyz18}
    Let $Z \subseteq \bbP^2$ be a set of $18$ general points and let $f_0,f_1,f_2 \in I(Z)_5 \subseteq S_5$ be linearly independent. Then $f_0,f_1, f_2$ have no quadratic syzygies. That is, $h_{S/I_{\langle 5 \rangle}}(7) = 18$.
\end{proposition}

We sketch two different proofs of \autoref{prop:noquadsyz18}, to illustrate the idea of the more general proofs of \autoref{thm: rmax in P2}, \autoref{thm:main3.3} and \autoref{thm: rmax in general}. A straightforward dimension count shows that $h_{S/I_{\langle 5 \rangle}}(7) = 18$ if and only if the forms $f_0,f_1,f_2$ do not have quadratic syzygies. The Hilbert-Burch matrix $B$ has the form 
\[
B \,=\, \begin{pmatrix}
    q_{00} & q_{01} & q_{02} \\
    q_{10} & q_{11} & q_{12} \\
    q_{20} & q_{21} & q_{22} \\
    \ell_0 & \ell_1 & \ell_2
\end{pmatrix}  \,\in\, S^{4 \times 3},
\]
where $q_{ij} \in S_2$ are quadrics, and $\ell_i \in S_1$ are linear forms. The quadratic syzygies of $f_0,f_1, f_2$ are the $\KK$-linear combinations of the columns of $B$ whose last entry is zero. If such a non-trivial $\KK$-linear combination exists, the linear forms $\ell_0, \ell_1, \ell_2$ are linearly dependent. Hence, $\rmV(\ell_0,\ell_1,\ell_2) \subseteq \mathbb{P}^2$ is non-empty. The quintics $f_0,f_1,f_2$ are the $3\times 3$ minors of $B$ involving the last row, so that $\rmV(\ell_0,\ell_1,\ell_2) \subseteq \rmV(f_0,f_1,f_2) $, showing $\rmV(\ell_0,\ell_1,\ell_2)$ is one of the points in $Z$. The genericity of $Z$, together with the fact that the construction is invariant under permutation of the $18$ points, leads to a contradiction; we refer to the proof of \autoref{thm:main3.3} for details on this construction. Hence $\ell_0, \ell_1, \ell_2$ are linearly independent, and one concludes that $f_0,f_1,f_2$ do not have quadratic syzygies. An analogous argument will give the proof of \autoref{thm:main3.3}.

Alternatively, one can prove \autoref{prop:noquadsyz18} via a classical liaison argument. Suppose $s_0f_0 + s_1f_1 + s_2f_2 = 0$ is a quadratic syzygy of $f_0,f_1,f_2$, for some $s_j \in S_2$. Assume $f_0,f_1,f_2$ are chosen generically in $I(Z)_5$. Let $K \subseteq \mathbb{P}^2$ be the set of points defined by the ideal $\langle f_1, f_2 \rangle_S$. By Bézout's theorem, and \autoref{lem:IZ regular sequences}, $K$ is a complete intersection of $25$ reduced points, and $Z$ is a subset of $K$. In other words, $K = Z \mathbin{\dot{\cup}} Z'$ where $Z'$ is a set of $7$ points, called the \emph{liaison of $Z$ in $K$}. For every  $z \in Z'$, we have $s_0(z)f_0(z) = 0$. We have $f_0(z) \neq 0$, otherwise $z \in Z$, and we conclude that $s_0 \in I(Z')$. On the other hand, the theory of liaison guarantees that $Z'$ has no nonzero quadratic equations; the necessary results to prove this statement are presented in \autoref{sec: bound rmax general}. We conclude that $s_0 = 0$, and analogously for $s_1=s_2=0$. This argument generalizes to a proof of \autoref{thm: rmax in P2} and \autoref{thm: rmax in general}.

\subsection{The case \texorpdfstring{$r = r_{d,\max}$}{r=r\_d,max}} \label{sec:rdmax}
In this section, we prove \autoref{conj:chopped_hilbert_function_intro} and \autoref{conj:main} for the maximal number of points $r = r_{d,\max}$ in the plane. Fix $d \geq 5$ and let $Z \in (\mathbb{P}^2)^{r_{d,\max}}$ be a general collection of $r_{d,\max}$ points. By construction, $I(Z)$ is zero in degree smaller than $d$ and the chopped ideal $I(Z)_{\langle d \rangle}$ has $3$ generators of degree $d$. By \autoref{thm:chopCutsZ}, $I(Z)_{\langle d \rangle}$ defines $Z$ scheme-theoretically. Since the IGC holds in $\bbP^2$, the three generators of degree $d$ have no linear syzygies and $I(Z)$ has $d-4$ minimal generators of degree $d+1$. The minimal free resolution of $I(Z)$ is
\begin{equation} \label{eq:freeresrmax}
\begin{tikzcd}[sep=scriptsize]
	0 & I(Z) \arrow[l] & {\!\!\begin{array}{c}S[-d]^{3}\\\bigoplus\\S[-(d+1)]^{d-4}\end{array}\!\!} \arrow[l] & S[-(d+2)]^{d-2}\arrow[l, "B"'] & 0. \arrow[l]
\end{tikzcd}
\end{equation}
\autoref{conj:chopped_hilbert_function_intro} predicts the value for $h_{I_{\langle d \rangle}}$ in degree $d+e$:
\[
h_{I(Z)_{\langle d \rangle}}(d+e) \,=\, \min\left\{3\cdot \binom{e+2}{2}, \binom{d+e+2}{2} - \binom{d+2}{2}+3\right\}.
\]
For $e = d-3$, both arguments give the minimum:
\[
3 \cdot \binom{d-1}{2} \, = \, \frac{3d^2-9d+6}{2} \, = \,  \binom{2d-1}{2} - \binom{d+2}{2} + 3.
\]
Hence, we expect the map $\mu_{d-3} \colon S_{d-3} \otimes I(Z)_d \to I(Z)_{2d-3}$ to be an isomorphism.

\begin{theorem}\label{thm: rmax in P2}
Let $Z \in (\mathbb{P}^2)^{r_{d,\max}}$ be a collection of $r_{d,\max}$ general points, and let $I_{\langle d \rangle} = \langle I(Z)_d \rangle$ be its chopped ideal. The Hilbert function of ${I_{\langle d \rangle}}$ satisfies
\begin{alignat*}{2}
h_{I_{\langle d \rangle}}(t) &\,=\, 0               & \quad & \text{if $t<d$,} \\
h_{I_{\langle d \rangle}}(t) &\,=\, 3\cdot h_S(t-d) & \quad & \text{if $d\leq t \leq 2d-3$,} \\
h_{I_{\langle d \rangle}}(t) &\,=\, h_S(t)- r_{d,\max} & \quad & \text{if $t \geq 2d-3$.} 
\end{alignat*}
In this case, \autoref{conj:chopped_hilbert_function_intro} and \autoref{conj:main} hold and $\gamma_2(d,r_{d,\max}) = d-3$.
\end{theorem}
\begin{proof}
    It suffices to show that the three generators $f_0,f_1,f_2$ of the chopped ideal do not have syzygies in degree $d-3$. This guarantees that the inequality in \autoref{thm: ESC lower bound} is a pointwise upper bound, and in turn that equality holds. Suppose $(s_0,s_1,s_2)$ is a syzygy of degree $d-3$:
    \[
    s_0f_0 + s_1f_1 + s_2 f_2 = 0 \qquad \text{for some } s_i \in S_{d-3}.
    \]
    We are going to prove that $s_0 = s_1=s_2 = 0$. By \autoref{lem:IZ regular sequences} we may assume that $f_1,f_2$ generate a complete intersection ideal defining a set $K$ of $d^2$ distinct points in $\bbP^2$. The set $K$ contains $Z$, and we write $Z' = K \setminus Z$ for the complement of $Z$ in $K$. 
    
    It suffices to show that $I(Z')_{d-3} = 0$. Indeed, for every $z \in Z'$ we have $f_1(z) = f_2(z) = 0$, which implies $f_0(z) s_0(z) = 0$. But $f_0(z) \neq 0$ because by \autoref{thm:chopCutsZ} the chopped ideal defines $Z$ scheme-theoretically and $z \notin Z$. Hence $s_0 \in I(Z')_{d-3}$. If $I(Z')_{d-3} = 0$, we obtain $s_0 = 0$. This implies $s_1 = s_2 = 0$ as well, because $s_1,s_2$ defines a syzygy of the complete intersection $\langle f_1,f_2 \rangle_{S}$, which does not have non-trivial syzygies in degree smaller than $d$.
    
    We are left with showing that $I(Z')_{d-3} = 0$. The Hilbert-Burch matrix $B'$ of $Z'$ can be obtained from the Hilbert-Burch matrix $B$ of $Z$ in \eqref{eq:freeresrmax} as follows; see, e.g., \cite[Prop.~1.3]{Sauer1985:PointsP2}. The entries of $B$ in its first three rows are quadrics, and the ones in the remaining $d-4$ rows are linear forms. The second and third row of $B$ correspond to $f_1, f_2$. The Hilbert-Burch matrix $B'$ is the transpose of the submatrix obtained from $B$ by removing the two rows corresponding to $f_1,f_2$. Therefore $B'$ is a $(d-2) \times (d-3)$ matrix whose first column consists of quadratic forms, and the remaining $d-4$ columns consist of linear forms. The maximal minors of $B'$ have degree $d-2$, and they are minimal generators of $I(Z')$ by the Hilbert-Burch Theorem. In particular, $I(Z')_{d-3} = 0$, as desired.
\end{proof}

\autoref{thm: rmax in P2} is a special version of  \autoref{thm: rmax in general}, whose proof resorts to liaison theory in higher dimension and is less explicit. Therefore, we chose to include both proofs. 

\autoref{thm: rmax in P2} allows us to provide an upper bound on the saturation gap for any set of general points in $\bbP^2$. Let $r \leq r_{d,\max}$ and fix $r$ general points $Z$ in $\mathbb{P}^2$. By definition, $\gamma_2(d,r) = \min \set{ e \in \mathbb{Z}_{>0} | h_{S/I_{\langle d \rangle}}(d+e) = r }$, where $I_{\langle d \rangle} = \langle I(Z)_d \rangle_S$ for $r$ general points $Z$.  
\begin{corollary} \label{cor:boundgap2}
For $r\leq r_{d,\max}$ general points in the plane, the saturation gap $\gamma_2(d,r)$ is at most $d-3$. In particular, the alternating sum in \autoref{conj:chopped_hilbert_function_intro} reduces to a single summand.
\end{corollary}
\autoref{cor:boundgap2} is a special case of \autoref{cor:boundgapgeneral} below.

\subsection{The case \texorpdfstring{$r = r_{d,{\min}}$}{r=d,min} for odd \texorpdfstring{$d$}{d}} \label{sec:rdmin}

Let $d = 2\delta+1$ be odd, and let $Z$ be a set of $r =  r_{d,\min} = 2(\delta+1)^2$ general points in $\bbP^2$. By \autoref{prop: MRC in P2}, $I(Z)$ is generated by $\delta+1$ forms of degree $d$ and $1$ form of degree $d+1$: 
\[
I(Z) \,=\, \langle f_0 \vvirg f_\delta , g \rangle _S,
\]
with $f_0 \vvirg f_\delta \in S_d$ and $g \in S_{d+1}$. In this section we prove the following result.

\begin{theorem} \label{thm:main3.3}
Let $d = 2 \delta + 1$ and let $Z \in (\mathbb{P}^2)^{r_{d,\min}}$ be a collection of $r_{d,\min} = 2(\delta+1)^2$ general points. The Hilbert Function of $I_{\langle d \rangle} = \langle I(Z)_d \rangle_S$ satisfies
\begin{align*}
    h_{S/I_{\langle d \rangle}}(d) \,&=\, r_{d,\min} , \\
    h_{S/I_{\langle d \rangle}}(d+1) \,&=\, r_{d,\min} +1 ,  \\
    h_{S/I_{\langle d \rangle}}(t) \,&=\, r_{d,\min} \phantom{+1}\quad\text{if $t \geq d+2$.} 
\end{align*}
In this case, \autoref{conj:chopped_hilbert_function_intro} and \autoref{conj:main} hold and $\gamma_2(d,r_{d,\min}) = 2$.
\end{theorem}

The proof uses the minimal free resolution of $I(Z)$, obtained from \autoref{prop: MRC in P2}
\[
\begin{tikzcd}[sep=scriptsize]
	0 & I(Z) \arrow[l] & {\!\!\begin{array}{c}S[-d]^{\delta+1}\\\bigoplus\\S[-(d+1)]\end{array}\!\!} \arrow[l] & S[-(d+2)]^{\delta+1}\arrow[l, "B"'] & 0. \arrow[l]
\end{tikzcd}
\]
The Hilbert-Burch matrix $B$ has the following form:  
\[
B \,=\, \left( \begin{array} {ccc}
q_{00} & \cdots & q_{0\delta } \\
\vdots & & \vdots \\
q_{\delta 0} & \cdots & q_{\delta\delta} \\
\ell_0 &  \cdots & \ell_\delta
\end{array}
\right),
\]
for some quadratic forms $q_{ij} \in S_2$ and linear forms $\ell_j \in S_1$. The degree $e$ syzygies of $f_0,\ldots, f_\delta$ are the elements of the $\KK$-vector space 
\[
{\rm Syz}(f_0,\ldots, f_\delta)_e \,=\, \Set{ (s_0, \ldots, s_\delta) \in (S_e)^{\delta+1} | s_0 f_0 + \dots + s_\delta f_\delta  =  0 }.
\]
\begin{proof}[Proof of \autoref{thm:main3.3}]
The fact that $h_{S/I_{\langle d \rangle}}(d) = h_{S/I}(d) = r_{d, \min}$ follows by \eqref{eq:genericHF}. The IGC holds for $n = 2$, and it implies $h_{S/I_{\langle d \rangle}}(d+1) = r_{d, \min} + 1$. 

For the statement on $h_{S/I{\langle d \rangle}}(t)$ for $t \geq d + 2$, observe that $(\delta + 1) \cdot h_S(2) = h_S(d+2) -  r_{d,\min} + ( \delta - 2)$.  Hence, it suffices to show that the forms $f_0 \vvirg f_\delta$ generating the chopped ideal $I_{\langle d \rangle} = \langle I(Z)_d \rangle_S$ have exactly $\delta-2$ quadratic syzygies, i.e.
\[ 
\dim_\KK \Syz (f_0, \ldots, f_\delta)_2 \,=\, \delta - 2.
\]
It is clear that $\dim_\KK {\rm Syz}(f_0, \ldots, f_\delta)_2 \geq  h_I(d)h_S(2) - h_I(d+2) = \delta - 2$; this also follows from \autoref{thm: ESC lower bound}. We show that $f_0, \ldots, f_\delta$ cannot have $\delta - 1$ syzygies.

The linear span $L_Z \coloneqq \langle \ell_0, \ldots, \ell_\delta \rangle_\KK$ of the linear forms in the last row of $B$ does not depend on the choice of the minimal free resolution. This is a consequence of \cite[Thm.~20.2]{Eis:CommutativeAlgebra}. 

If $f_0 \vvirg f_\delta$ have $\delta-1$ quadratic syzygies for a general choice of $Z$, then $\dim_\KK L_Z$ is at most 2, which implies that the variety $\rmV(L_Z)$ is non-empty. Moreover, by the Hilbert-Burch Theorem, the ideal generated by $L_Z$ contains $I_{\langle d \rangle} = \langle f_0 \vvirg f_\delta \rangle_S$, which cuts out $Z$ scheme-theoretically by \autoref{thm:chop_map}. Hence $\dim_\KK L_Z = 2$ and $\rmV(L_Z)$ is one of the points in $Z$. Define
\[
\psi \colon (\bbP^2 )^{r} \dashto \bbP^2, \qquad \psi(Z) \coloneqq \rmV(L_Z).
\]
This is a rational map: the coordinates of $L_Z$, and hence of the point it defines, can be expressed as polynomial functions of the coordinates of the configuration $Z \in (\bbP^2 )^{r}$, for instance via successive applications of Cramer's rule. Moreover, we showed that $\psi(Z) \in Z$. By definition, $\psi$ is invariant under the action of the symmetric group $\frakS_r$ permuting the factors of $(\bbP^2 )^{ r}$. Consider the subvarieties of $(\bbP^2 )^{r}$ defined by 
\[
Y_j \,\coloneqq\, \overline{\Set{ Z =(z_1 \vvirg z_r) \in \operatorname{Dom}(\psi) | \psi(Z) = z_j } }.
\]
We have $(\bbP^2 )^{r} = \bigcup_{j=1}^r Y_j$. Since $(\bbP^2 )^{r}$ is irreducible, we have $Y_j = (\bbP^2 )^{ r}$ for at least one $j$. On the other hand, $\frakS_r$ invariance implies that if $Y_j = (\bbP^2 )^{r}$ for one $j$, then this must be true for all $j$'s. But any two $Y_j$ are distinct because generically $Z$ consists of distinct elements. This gives a contradiction showing that the map $\psi$ cannot exist. We obtain that $\dim_\KK L_Z = 3$ for general $Z$ and this concludes the proof.
\end{proof}

\section{The largest possible saturation gap} \label{sec: bound rmax general}

In this section, we consider a set $Z$ of $r = \binom{d+n}{n} - (n+1)$ general points in $\bbP^n$. \autoref{thm:chopCutsZ} guarantees this is the largest possible number of points such that the chopped ideal $I_{\langle d \rangle} = \langle I(Z)_d \rangle_S$ in the ring $S = \KK[x_0, \ldots, x_n]$ defines $Z$ scheme-theoretically. We will show that \autoref{conj:chopped_hilbert_function_intro} is true in this case:
\begin{theorem}\label{thm: rmax in general}
   Let $n, d$ be positive integers and let $Z \subseteq \bbP^n$ be a set of $r =  h_S(d) - (n+1)$ general points. The Hilbert function of the chopped ideal $I_{\langle d \rangle}= \langle I(Z)_d \rangle_S$ satisfies 
   \[
   h_{S/{I_{\langle d \rangle}}} ( d+e ) \,=\, \sum_{k \geq 0} (-1)^{k} \cdot h_S(d+e - k d) \cdot \binom{n+1}{k} ,
   \]
   for $e \leq (n-1)d - (n+1)$. \autoref{conj:chopped_hilbert_function_intro} and \autoref{conj:main} hold with $\gamma_n(d,r) = (n-1)d-(n+1)$.
\end{theorem}
As a consequence of \autoref{thm: rmax in general}, we will obtain an upper bound on the saturation gap $\gamma_n(d,r)$ for all $r < h_S(d)-n$ in \autoref{cor:boundgapgeneral}.

The proof of \autoref{thm: rmax in general} relies on a fundamental fact in the theory of liaison. Given the Hilbert function of a set of points $Z$, one can compute the Hilbert function of the complementary set of points $Z' = K \setminus Z$ in a complete intersection $K \supseteq Z$. We record this fact in \autoref{prop: DhZ v DhZ'} below. In order to state it precisely, we introduce the following notation. The function $\Delta h_Z(t) = h_Z(t) - h_Z(t-1)$, with $h_Z = h_{S/I(Z)}$, is the \emph{first difference} of the Hilbert function of $Z$. Often, $\Delta h_Z(t)$ is called the $h$-vector of $Z$ and it is recorded as the sequence of its non-zero values. We record some immediate properties, see e.g. \cite{Chi:HilbertFunctionsTensorAn}.
\begin{lemma}\label{lemma: basic DhZ}
   For a finite set of points $Z \subseteq \bbP^n$, we have 
   \begin{enumerate}[label=\textup{(\roman*)}]
       \item $h_Z(t) = \sum_{t'=0}^t \Delta h_Z(t')$;
       \item $\regH(Z) = \max\set{ t | \Delta h_Z(t) > 0 }$; 
       \item if $I(Z)_t = 0$ then $\Delta h_Z(t') = \binom{t'+n-1}{n-1}$ for $t' \leq t$.
   \end{enumerate}
\end{lemma}

For complete intersections $K$, the function $\Delta h_K$ has a symmetry property \cite[Ch.~17]{Eis:CommutativeAlgebra}.
\begin{lemma}\label{lemma: CI ideal}
    Let $K \subseteq \bbP^n$ be a set of $d_1 \cdots d_n$ points whose ideal is generated by a regular sequence $f_1 \vvirg f_n$, with $\deg(f_i) = d_i$. Then $\regH(K) = d_1 + \cdots + d_n - n$. In particular, if $d_1 = \cdots = d_n = d$, $\regH(K) = n(d-1)$. Moreover, $\Delta h_K$ is symmetric, that is, setting $\rho = \regH(K)$, we have $\Delta h_K(t) = \Delta h_K(\rho - t)$.
\end{lemma}

The theory of liaison studies the relation between the distinct irreducible components (or union of such) of a complete intersection; we only illustrate one result in the context of ideals of points; we refer to \cite{MigNag:LiaisonRelatedTopics,Mig:IntroLiaison} for an extensive exposition of the subject. The following is a consequence of \cite[Prop.~5.2.1]{Mig:IntroLiaison}; see also \cite[Eqn.~(3)]{AngChi:DescriptionQuartics}.
\begin{proposition}\label{prop: DhZ v DhZ'}
Let $Z \subseteq \bbP^n$ be a set of points and let $f_1 \vvirg f_n \in I(Z)$ be homogeneous polynomials of degree $d_1, \ldots, d_n$ defining a smooth complete intersection $K$ of degree $d_1 \cdots d_n$. Let $Z' = K \setminus Z$ be the complement of $Z$ in $K$. Let $\rho = \regH(K) = d_1 + \cdots + d_n - n$. Then $\Delta h_{Z}(t) + \Delta h_{Z'}(\rho - t) = \Delta h_K(t)$.
\end{proposition}
In words, \autoref{prop: DhZ v DhZ'} says that the sequence $\Delta h_{Z'}$ equals the sequence $\Delta h_{K} - \Delta h_Z$, in the reversed order.

The proof of this result uses a construction known as the \emph{mapping cone} in homological algebra. The key fact is that the resolution of $I(Z')$ can be obtained from that of $I(Z)$ using the fact that the resolution of $I(K)$ is the classical Koszul complex of $f_1,\ldots, f_n$ \cite{Mig:IntroLiaison}.

We now have all the ingredients to give a proof of \autoref{thm: rmax in general}.

\begin{proof}[Proof of \autoref{thm: rmax in general}]
By construction, we have $\dim_\KK (I_{\langle d \rangle})_d = \dim_\KK I(Z)_d = n+1$. Removing $z_0 \in Z$, a general set $f_1,\dots,f_n \in I(Z\setminus\{z_0\})_d$ defines a reduced complete intersection $K \subseteq \bbP^n$ of $d^n$ points by \autoref{lem:IZ regular sequences}. We add a generator $f_0$ such that $I_{\langle d \rangle} = \langle f_0, \dots, f_n \rangle_{S}$.

The Hilbert function $h_K(t)$ has the following compact form, which can be computed directly from the dimension of the syzygy modules in the Koszul complex \cite[Ch.~17]{Eis:CommutativeAlgebra}:
\[
h_K(t) \,=\, \sum_{k = 0}^n (-1)^k h_S(t-kd) \binom{n}{k}.
\]
In particular $h_K(t) = h_Z(t)$ for $t \leq d-1$, and $h_K(d) = h_Z(d)+1$. Recall from \autoref{lemma: CI ideal} that $\regH(K) = n(d-1)$. Set $\rho \coloneqq n(d-1)$ and let $Z' \coloneqq K \setminus Z$ be the complement of $Z$ in $K$. By \autoref{lemma: CI ideal} and \autoref{prop: DhZ v DhZ'}, we have 
\[
\Delta h_Z(t) + \Delta h_{Z'}(\rho - t) \,=\, \Delta h_K(t) \,=\, \Delta h_K(\rho - t).
\] 
Since $\Delta h_Z(t) = 0$ for $t \geq d+1$, we deduce $\Delta h_{Z'}(t) = \Delta h_{K}(t)$ for $t \leq \rho-(d+1)$.  Since $I(Z') \subseteq I(K)$, this implies $I(Z')_t = I(K)_t$ for $t \leq \rho - (d+1)$. In particular, we obtain that 
\[
\dim_\KK I(Z')_{t}  \,=\, h_S(t) - h_K(t) \,=\, \sum_{k \geq 1} (-1)^{k+1} h_S(t-kd) \binom{n}{k},
\]
for $t \leq (n-1)d - (n+1)$. See \autoref{fig:pretty HF} for a schematic representation.

\begin{figure}[!ht]
    \centering
    \includegraphics[width=12cm]{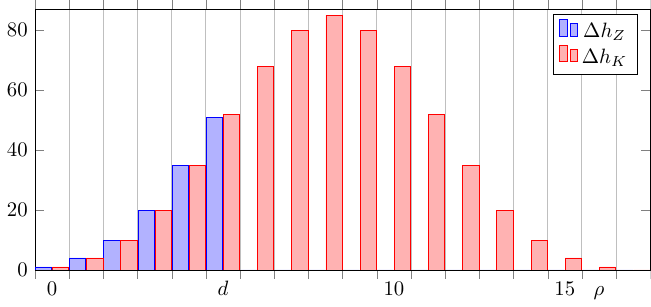}
    \caption{The sequences $\Delta h_Z$ and $\Delta h_K$ for $r=121$ points in $\PP^4$. Their difference, in reversed order, equals $\Delta h_{Z'}$. Note $\Delta h_K(5) = \Delta h_Z(5)+1$.
    }\label{fig:pretty HF}
\end{figure}

Now, fix $e \leq nd-n-(d+1)$ and let $\Syz_{e} \subseteq S_e^{n+1}$ be the subspace of syzygies of degree $e$; i.e.\ the tuples $ (s_0,\dots,s_n)$ with $s_0f_0 + \cdots + s_nf_n = 0$.

Consider the projection $\pi_0 \colon \Syz_e \to S_e$ onto the $0$-th component $S_e$, that is $\pi(s_0,\dots,s_n) = s_0$. We observe that the image of this map is contained in $I(Z')$. Indeed, since $Z' \subseteq K$, for every $p \in Z'$ we have $f_j(p) = 0$ for $j =1 \vvirg n$. This implies $s_0(p)f_0(p) = 0$ for $p \in Z'$. Since $I_{\langle d \rangle}$ defines $Z$ scheme-theoretically, we deduce that $f_0(p) \neq 0$ for every $p \in Z'$; hence $s_0(p) = 0$ for $p \in Z'$, as desired. Therefore, the image of the map $\pi_0$ is contained in $I(Z')_{e}$.

The kernel of $\pi_0$ consists of tuples $(0, s_1 \vvirg s_n) \in S_e^{n+1}$ such that $s_1f_1 + \cdots + s_nf_n = 0$; hence $(s_1 \vvirg s_n)$ defines a syzygy of $f_1 \vvirg f_n$. Since the ideal $\langle f_1 \vvirg f_n \rangle_S$ is a complete intersection, its only syzygies in degree $e$ are generated by the Koszul syzygies, and we deduce  
\[
\dim_\KK  (\Ker \pi_0) \,=\, \sum_{k \geq 2} (-1)^{k} h_S(d+e  - kd)\binom{n}{k} .
\]
Since $\dim_\KK \Syz_e = \dim_\KK (\Ker \pi_0) + \dim_\KK (\operatorname{Im} \pi_0) \leq \dim_\KK (\Ker \pi_0) + \dim_\KK I(Z')_{e}$, the standard identity $\binom{n}{k} + \binom{n}{k+1} = \binom{n+1}{k+1}$ yields
\begin{align*}
\dim_\KK \Syz_e &\,\leq\, \sum_{k \geq 1} (-1)^{k+1} h_S(e - kd) \binom{n}{k+1}  + 
   \sum_{k \geq 1} (-1)^{k+1} h_S(e-kd) \binom{n}{k}  \\
&\,=\,\sum_{k \geq 1} (-1)^{k+1} h_S(e-kd) \binom{n+1}{k+1}.
\end{align*}
We conclude that for $e \leq (n-1) d - (n+1)$,
\[
h_{I_{\langle d \rangle}}(d+e) \,=\,  (n+1) \cdot h_S(e) - \dim_\KK \Syz_{e} \,\geq\,  \sum_{k \geq 1} (-1)^{k+1} h_S(d+e-kd) \binom{n+1}{k}.
\]
The reverse inequality for $h_{I_{\langle d \rangle}}(d+e)$ follows from Fröberg's \autoref{thm:fröberg} which is a pointwise lower bound since Fröberg's conjecture is true in this case \cite[Section 3.1]{Fro:IneqHilbertSer}.

It remains to show that the inequality $h_{S/I_{\langle d \rangle}}(d+e) \geq h_S(d+e)-(n+1)$ is an equality for $e_0 \coloneqq (n-1)d - (n+1)$ and strict for $e_0-1$ (\autoref{thm: ESC lower bound}), thus establishing the length of the saturation gap $e_0$.
We have shown that $h_{S/I_{\langle d \rangle}}(t) = \text{frö}_{d,n+1}(t)$ for $t\leq nd-(n+1)$, where the latter is the Hilbert function of any ideal generated by a regular sequence of $n+1$ elements in degree $d$. Hence it suffices to check the (in)equalities for any such complete intersection, for example the monomial ideal $M \coloneqq \langle x_0^d, \dots, x_n^d\rangle_S$.
The Hilbert function of $S/M$ has the combinatorial description
\[
h_{S/M}(t) \,=\, \# \Set{\bm{x}^\alpha \in S_t | \bm{x}^\alpha \text{ divides } x_0^{d-1}\dotsm x_n^{d-1} }.
\]
The involution $\bm{x}^\alpha \mapsto x_0^{d-1}\dotsm x_n^{d-1}/\bm{x}^\alpha$ shows the symmetry $h_{S/M}(t) = h_{S/M}((n+1)(d-1) - t)$. Thus
\begin{align*}
h_{S/M}(d+e_0) &\overset{\phantom{\text{sym}}}{=} h_{S/M}(n(d-1) - 1) \overset{\text{sym}}{=} h_{S/M}(d) \,=\, h_S(d)-(n+1) \\
h_{S/M}(d+e_0-1) &\overset{\text{sym}}{=} h_{S/M}(d+1) \,=\, h_S(d+1) - (n+1)n
\end{align*}
using $d\geq 2$ in the last equality. To see that the last term is strictly larger than $h_S(d)-(n+1)$, we calculate
\begin{gather*}
\big(h_S(d+1) - (n+1)n\big) - \big(h_S(d) - (n+1)\big) \,=\, \binom{n+d}{n-1} - (n+1)(n-1) \\
\overset{d\geq 2}{\geq}\, \frac{(n+2)(n+1)n}{6} - (n+1)(n-1) \,=\, \frac{1}{6}(n+1)(n^2-4n+6) \,>\, 0. \qedhere
\end{gather*}
\end{proof}

This leads to the following generalization of \autoref{cor:boundgap2}.
\begin{corollary}\label{cor:boundgapgeneral}
   For $r\leq h_S(d)-(n+1)$ general points in the plane, the saturation gap $\gamma_n(d,r)$ is at most $(n-1)d-(n+1)$. The sum in \autoref{conj:chopped_hilbert_function_intro} reduces to $n-1$ terms.
\end{corollary}
\begin{proof}
The proof is by reverse induction on $r$. The base case is $r = h_S(d)-(n+1)$, for which the statement follows from \autoref{thm: rmax in general}. We are going to show that $\gamma_n(d,r-1) \leq \gamma_n(d,r)$.
Let $Z_{r-1} = (z_1, \ldots, z_{r-1}) \in (\mathbb{P}^n)^{r-1}$ be general, let $z_r$ be one additional general point and set $Z_r = (z_1, \ldots, z_r)$. Let $e_0 \coloneqq \gamma_n(d,r)$, we have $h_{S/I(Z_r)_{\langle d \rangle}}(d+e_0) = r$ and need to show $h_{S/I(Z_{r-1})_{\langle d \rangle}}(d+e_0) = r-1$. For this is suffices to show
\[
(I(Z_{r-1})_{\langle d \rangle})_{d+e_0} \,=\, S_{e_0} \cdot I(Z_{r-1})_d \,\supsetneq\,  S_{e_0} \cdot I(Z_{r})_d \,=\, (I(Z_{r})_{\langle d \rangle})_{d+e_0}.
\]
By genericity of $Z_r$, we can pick $f \in I(Z_{r-1})_d \setminus I(Z_r)_d$ and $h \in S_{e_0}$ not vanishing on $z_r$, then $fh \in S_{e_0} \cdot I(Z_{r-1})_d$, but $fh \notin I(Z_r) \supseteq I(Z_{r})_{\langle d \rangle}$.
\end{proof}

\section{Proofs via computer algebra} \label{sec:5}

This section provides a computational proof of \autoref{conj:chopped_hilbert_function_intro} for many small values of $d,n,r$. 

\begin{theorem} \label{thm:computeralgebra}
\autoref{conj:chopped_hilbert_function_intro} holds for a set of $r$ general points in $\bbP^n$, with the following values of $n$ and $r$:
 \[
\begin{array}{c|ccccccccc}
n  & 2 & 3  & 4 & 5 & 6 & 7 & 8 & 9 & 10 \\ \hline
r  &  \leq 2343 & \leq 2296 & \leq 1815 & \leq 1272 & \leq 908 & \leq 767 & \leq 479 & \leq 207 & \leq 267  
\end{array}
 \]
\end{theorem}

The proof of  \autoref{thm:computeralgebra} is computational. For every such $(n,r)$ of interest, we exhibit an $r$-tuple $Z \in (\mathbb{P}^n(\mathbb{Q}))^r$ for which the statement holds. The following semicontinuity result guarantees that this suffices to conclude. 

\begin{proposition} \label{prop: semicontinuity multiplication map}
Let $r < h_S(d)-n$ and let $U_{\genHF} \subseteq (\mathbb{P}^n)^r$ be the collections of points satisfying \eqref{eq:genericHF}. The set  $U_{k,e} = \set{ Z \in U_{\mathrm{genHF}} | h_{I(Z)_{\langle d \rangle}}(d+e) \geq k }$ is Zariski open in $(\bbP^n)^r$. In particular, the set $U = \set{ Z \in (\bbP^n)^r | h_{I(Z)_{\langle d \rangle}} \text{ satisfies \eqref{eq:expectedhf}}}$ is Zariski open. 
\end{proposition}
\begin{proof}
   Let $\Gr( h_S(d) - r, S_d)$ be the Grassmannian of planes of dimension $ h_S(d) - r$ in $S_d$. Consider the vector bundle $\calE \coloneqq \Hom(\calS \otimes S_e,  S_{d+e})$, where $\calS$ denotes the tautological bundle over $\Gr( h_S(d) - r, S_d)$: the fiber of $\calE$ at a plane $[L]$ is $\calE_L = \Hom( L \otimes S_e, S_{d+e})$. The multiplication map $\mu: S_d \otimes S_e \to S_{d+e}$ defines a global section of $\calE$ via restriction. The pull-back $\frakc^* \calE$ of $\calE$ via the chopping map (\autoref{def:chopping_map}) defines a bundle over $U_{\text{genHF}}$ and $\frakc^* \mu$ defines a global section of $\frakc^*\calE$. The set $U_{k,e}$ is the complement of the degeneracy locus 
   \[
   V_{k,e} \,\coloneqq\, \Set{ Z \in U_{\textup{genHF}} | \rank\bigl(\mu_e\colon I(Z)_d \otimes S_e \to S_{d+e} \bigr) < k }.
   \]
   This shows that $V_{k,e}$ is Zariski closed, hence $U_{k,e}$ is Zariski open.
   
   The set $U$ is the intersection of  $U_{\text{genHF}}$ with finitely many open sets $U_{k_e,e}$ for $e = 1 \vvirg m_r$. Here $m_r$ is any upper bound on the saturation gap, for instance $m_r = (n-1)d-(n+1)$ from \autoref{cor:boundgapgeneral}; $k_e$ is the expected value of $h_{I_{\langle d \rangle}}(d+e)$ in \eqref{eq:expectedhf}.  This concludes the proof.
\end{proof}
\autoref{thm:computeralgebra} is a direct consequence of \autoref{prop: semicontinuity multiplication map}.
\begin{proof}[Proof of \autoref{thm:computeralgebra}]
Identify the field of rational numbers $\bbQ$ with the prime field of $\KK$. For every $(n,r)$ of interest, we exhibit an instance $Z \in U_{\genHF} \subseteq (\mathbb{P}^n(\mathbb{Q}))^r$ for which the Hilbert function of $I(Z)_{\langle d \rangle}$ satisfies \eqref{eq:expectedhf}. These instances can be found online at 
\begin{center}
    \url{https://mathrepo.mis.mpg.de/ChoppedIdeals/}.
\end{center}
This guarantees that the corresponding open set $U$ of \autoref{prop: semicontinuity multiplication map} is non-empty, and therefore it is Zariski dense. This shows that for a general instance $Z \in (\bbP^n)^r$, the Hilbert function of $Z$ satisfies \eqref{eq:expectedhf}, and therefore \autoref{conj:chopped_hilbert_function_intro} holds. 
\end{proof}

Notice that it suffices to check \eqref{eq:expectedhf} for $e$ up to the expected saturation gap of \autoref{conj:main}. Indeed, if it holds up to that value, we have $(I_{\langle d \rangle})_{d+e} = I(Z)_{d+e}$ for higher $e$, which is enough to conclude.

To speed up the computations, we make the following observation. Let $I \subseteq S = \KK[x_0, \ldots, x_n]$ be an ideal generated by polynomials $f_1, \ldots, f_s$ with coefficients in $\ZZ \subseteq \KK$. Then 
\[
\dim_\KK I_t \,=\, \dim_\QQ (I \cap \QQ[x_0, \ldots, x_n])_t \,\geq\, \dim_{\mathbb{F}_p} (I_{\mathbb{F}_p})_t.
\]
Here $I_{\mathbb{F}_p} \subseteq \mathbb{F}_p[x_0, \ldots, x_n]$ is the reduction modulo $p$ of the ideal $I$. Checking that the upper bound \eqref{eq:expectedhf} is attained is much easier over a finite field, and it guarantees the bound is attained over $\mathbb{Q}$, hence over $\KK$. This leads to the following strategy for proving \autoref{thm:computeralgebra}. First implement the expected Hilbert function and the expected gap size according to \autoref{conj:chopped_hilbert_function_intro} and \autoref{conj:main}, here called $\texttt{expectedHF}(n,r,t)$ and $\texttt{expectedGapSize}(n,r)$. Next, execute the following steps for given $(n,r)$:

\begin{enumerate}
\item Determine $d\coloneqq \min \set{t | h_S(t) \geq r}$.
\item Sample $r$ points $Z\subseteq \PP^n(\FF_p)$ (represented by homogeneous integer coordinates).
\item Calculate the ideal $I \coloneqq I(Z)$ and set $J \coloneqq \langle I_d \rangle_{\FF_p[x_0,\dots,x_n]}$.
\item Calculate the Hilbert function of $S/J$ up to $d+\texttt{expectedGapSize}(n,r)$.
\item Check if the values match with $\texttt{expectedHF}(n,r,t)$.
\end{enumerate}
\autoref{prop: semicontinuity multiplication map} and the preceding remark about reduction modulo $p$ ensure that this procedure \emph{proves} the validity of \autoref{conj:chopped_hilbert_function_intro} in the cases of \autoref{thm:computeralgebra}. The following code in \texttt{Macaulay2}, assuming an implementation of \texttt{expectedGapSize} and \texttt{expectedHF}, demonstrates the procedure.

\begin{minted}{macaulay2.py:Macaulay2Lexer -x}
loadPackage "Points"
n = 2; r = 41;
d = 8; -- determined by (n,r)
I = points randomPointsMat(ZZ/1009[x_0..x_n], r);
e = expectedGapSize(n,r) -- |$\mathit 3$|
J = ideal select(first entries gens I, f -> degree f == {d});
hs = hilbertSeries(J, Order=>d+e+1) -- |$\mathit 1+3T+6T^2+10T^3+15T^4+21T^5+28T^6+36T^7+41T^8+43T^9+42T^{10}+41T^{11}$|
T = (ring hs)_0;
for t to d+e do
    assert (coefficient(T^t, hs) == expectedHF(n,r,t))
\end{minted}

We briefly discuss some additional speed-ups. In the cases outside of the range \eqref{eq:range_r}, where \autoref{conj:chopped_hilbert_function_intro} is equivalent to the IGC, it is \emph{much} faster to calculate $\texttt{mingens}(I(Z))$ and compare with the expected first graded Betti numbers $\beta_{1,d}$, $\beta_{1,d+1}$. Also, much computation time is spent computing the ideal of points. For large $r$, a significant speedup is obtained when using the methods implemented in \texttt{Points.m2} \cite{PointsSource}.

We conclude with a variation on the computer experiment. Instead of computing ideals of (random) points one can also try to find monomial ideals certifying \autoref{conj:chopped_hilbert_function_intro}. This method, for instance, can be used to prove the MRC in $\bbP^2$ \cite{Sauer1985:PointsP2,GerGreRob:MonomialIdeals}. An exhaustive search is possible in $\PP^2$ for small values of $r$ and leads to the following result:

\begin{theorem}\label{thm:nomonomials}
Let $n=2$, $S=\KK[x,y,z]$.
\begin{enumerate}[label=\textup{(\roman*)}]
\item For $r=18$ there is a unique, up to permutation of the variables, monomial ideal $I=\langle x^3y^2, y^3z^2, z^3x^2, x^2y^2z^2\rangle_S$ with Hilbert function \eqref{eq:genericHF}, which satisfies \autoref{conj:chopped_hilbert_function_intro}.
\item For $r\in\{25,32,33\}$ there are \emph{no} monomial ideals satisfying \autoref{conj:chopped_hilbert_function_intro}.
\end{enumerate}
\end{theorem}

\section{Symmetric tensor decomposition} \label{sec:6}

This final section discusses the role of chopped ideals in tensor decomposition algorithms. This was our original motivation and this project was initiated by a question encountered by one of the authors and Nick Vannieuwenhoven in \cite{TelVan:NormalFormAlgorithm}. The setting in that paper is slightly different because it studies algorithms for (ordinary) tensor decomposition. The same approach is classical in \emph{symmetric tensor decomposition} or \emph{Waring decomposition} \cite{IarrKan:PowerSumsBook,brachat2010symmetric}. A \emph{Waring decomposition} of a homogeneous polynomial $F \in T = \KK[y_0 \vvirg y_n]$ of degree $D$ is an expression of $F$ as a sum 
\begin{equation} \label{eq:waring}
    F \, = \, c_1 (z_1 \cdot y)^D + \cdots + c_r (z_r \cdot y)^D
\end{equation} 
of powers of linear forms. Here $c_i \in \KK$ are constants and $z_i \cdot y = z_{i0}y_0 + \cdots + z_{in}y_n$. The \emph{Waring rank} of $F$ is the minimal number of summands $r$ in such an expression. 

Most decomposition algorithms aim to determine the vectors $z_i$ up to scaling, and then solve a linear system to find the coefficients $c_i$. Therefore, it is natural to regard $Z = (z_1, \ldots, z_r)$ as a point in $(\PP^n)^r$. A classical approach to compute $Z$ uses \emph{apolarity theory} and dates back to \cite{Sylv:PrinciplesCalculusForms}. We briefly recall the basics. 

The ring $S = \KK[x_0\vvirg x_n]$ acts on the ring $T$ by differentiation: if $g \in S$ and $F \in T$, then $g \bullet F = g(\partial_0 \vvirg \partial_n) F$ where $\partial_j = \frac{\partial}{\partial y_j}$. This action is graded. In particular, every $F \in T_D$ gives rise to a sequence of linear maps
\[
C_F(d,D-d)\colon S_d \to T_{D-d}, \qquad g \mapsto g \bullet F,
\]
called the \emph{Catalecticant maps} of $F$. Notice that $\Ker C_F(d,D-d) \subseteq S_d$ is a linear space of polynomials of degree $d$, and that $C_F(d,D-d) = 0$ if $d > D$. The kernels $\Ker C_F(d,D-d)$ are the homogeneous components of an ideal, called the \emph{apolar ideal} of $F$, given by 
\[
 \Ann(F) \,=\, \Set{ f \in S | f \bullet F = 0 }.
\]
On the other hand, $S$ can be naturally regarded as the homogeneous coordinate ring of $\bbP^n$. The classical \emph{apolarity lemma} \cite[Sec.~1, Lem.~5]{bernardi2018hitchhiker} states that $F$ decomposes as in \eqref{eq:waring} if and only if the vanishing ideal $I(Z)$ of $Z = (z_1 \vvirg z_r)$ is contained in $\Ann(F)$.

It is usually hard to compute the ideal $I(Z)$ of a minimal Waring decomposition of $F$. However, in a restricted range, it turns out that its chopped ideal is generated by the graded component of $\Ann(F)$ in degree $d = \regH(Z)$. In other words, the chopped ideal of $Z$ can be computed via elementary linear algebra as kernel of the corresponding catalecticant map. This is known as the catalecticant method to determine a decomposition of $F$ and it is the starting point of a number of more advanced Waring decomposition algorithms \cite{BerGimIda:ComputingSymmetricRankSymmetricTensors,brachat2010symmetric,BerTau:WaringTangCactusDecomp,LafMasRis:DecompAlgs}. We record a consequence of \cite[Thm.~2.6, Lem.~1.19]{IarrKan:PowerSumsBook}.

\begin{theorem}\label{thm: tensor decomp}
 Let $D \geq 2d$. If $F \in T_D$ is a general form of rank $r < h_S(d) - n$, then
 \begin{enumerate}[label=\textup{(\roman*)}]
  \item there is a unique Waring decomposition $Z \subseteq (\bbP^n)^r$ of length $r$ and
  \item $\Ker C_F(d,D-d) = I(Z)_d $ generates the chopped ideal $I(Z)_{\langle d\rangle}$.
 \end{enumerate}
\end{theorem}
This suggests a strategy for computing the Waring decomposition \eqref{eq:waring} of $F \in T_D$, compare \cite[Alg.~2.82]{bernardi2018hitchhiker}:
\begin{enumerate}
    \item Construct the catalecticant matrix $C_F(d,D-d)$, with $d = \lfloor \frac{D}{2} \rfloor$.
    \item Compute a basis $f_1, \ldots, f_s$ for the kernel of $C_F(d,D-d)$ using linear algebra over $\KK$,
    \item Solve the polynomial system $f_1 = \dots = f_s = 0$ on $\mathbb{P}^n$. Let $Z = (z_1, \ldots, z_r) \in (\KK^{n+1})^r$ be the tuple of homogeneous coordinate vectors for the solutions.
    \item Solve the linear equations \eqref{eq:waring} for $c_1, \ldots, c_r$.
\end{enumerate}
When the rank of $F$ is at most $h_S(\lfloor D/2 \rfloor) - (n+1)$, and under suitable genericity assumptions, \autoref{thm: tensor decomp} guarantees that this method computes the unique Waring decomposition of $F$. Moreover, if one knows $r < h_S(d) - n$ for some $d < D/2$, the approach can be made more efficient by computing a smaller catalecticant matrix.

\begin{example}\label{ex: 18 points tensor}
    Consider a general ternary form of degree $D = 10$ with Waring rank $r =18$,
    \[
    F \,=\, (z_1 \cdot y)^{10} + \cdots + (z_{18} \cdot y)^{10}.
    \]
    Here $y = (y_0,y_1,y_2)$ and each $z_i$ has three coordinates as well. The catalecticant matrix $C_F(5,5)$ is of size $21 \times 21$, and has rank $18$. Its kernel consists of three ternary quintics in the variables $x_0, x_1, x_2$ passing through the $18$ prescribed points $z_1, \ldots, z_{18}$. They generate the chopped ideal $I_{\langle 5 \rangle} = \langle I(Z)_5 \rangle$ investigated in \autoref{sec:quintics}. 
\end{example}

The main work in this strategy is step 3: solving the polynomial system $f_1 = \cdots = f_s = 0$. If $\KK = \mathbb{C}$, two important strategies for doing this numerically are \emph{homotopy continuation} \cite{timme2021numerical} and \emph{numerical normal form methods} \cite{telen2020thesis}. We argue that in this setting it is natural to use the latter type of methods. Indeed, by construction, the system has $s = h_S(d) - r > n$ equations and $n+1$ variables, hence it is overdetermined. In homotopy continuation, this is typically dealt with by solving a square subsystem of $n$ equations which has, by B\'ezout's theorem, $d^n > r$ solutions. These candidate solutions are filtered by checking if all remaining equations also vanish. However, computing all these $d^n$ solutions becomes quickly infeasible. More refined algorithms using homotopy continuation are proposed in \cite{BerDalHauMou:TensorDecompHomCont}, but they rely on the knowledge of certain information on secant varieties which is out of reach with current methods. On the contrary, numerical normal form methods work directly with the overdetermined system, see \cite[Sec.~4.4]{bender2021yet}. A second advantage is that, while homotopy continuation requires $\KK = \mathbb{C}$, normal form methods work over any algebraically closed field $\KK$. 

Numerical normal form methods such as \cite{bender2021yet} and \cite[Sec.~4.5]{telen2020thesis} compute the points $z_i$ via the eigenvalues and/or eigenvectors of pairwise commuting \emph{multiplication matrices}. These are in turn computed from a different matrix $M(d+e)$, called \emph{Macaulay matrix}. Here $e \geq \gamma_{n,d}(r)$ is a positive integer for which $h_{S/I_{\langle d \rangle}}(d+e) = r$: the number of rows of $M(d+e)$ is $h_S(d+e)$, and its column span is $I(Z)_{d+e}$. In particular, upper bounds on the saturation gap allow one to work with the smallest admissible value $e_0$. We summarize the relation between Waring decomposition and \autoref{conj:main} as follows: 
\begin{center}
    \textit{The complexity of computing multiplication matrices in our setting is governed by linear algebra computations with the Macaulay matrix $M(d+e_0)$, where $e_0 = \gamma_{n,d}(r)$.}
\end{center}
To illustrate this punchline, we implemented the catalecticant algorithm in the Julia package \texttt{Catalecticant.jl}. It uses \texttt{EigenvalueSolver.jl}, a general purpose equation solver from \cite{bender2021yet}. Here is how to construct and decompose $F$ from \autoref{ex: 18 points tensor}:
\begin{minted}{julia}
n = 2; D = 10; r = 18; # define the parameters
@polyvar y[1:n+1] # variables of F
Z = exp.(2*pi*im*rand(r, n+1)) # generate random points Z 
F = sum((Z*y).^D)
c, linforms = waring_decomposition(F,y,r) # decompose F
\end{minted}
Here line 3 draws the coordinates of the points $Z$ uniformly from the unit circle in the complex plane; this avoids bad numerical behavior in the expansion $(z_i \cdot y)^D$. The output at line 5 is the pair of coefficients of the linear forms from \eqref{eq:waring}. This method assumes \autoref{conj:main}: \texttt{EigenvalueSolver.jl} constructs the Macaulay matrix $M(d+e_0)$, where $e_0 = \gamma_{n,d}(r)$ is the saturation gap predicted in \autoref{conj:main}. In this specific case, we have $d+e_0 = 7$, as illustrated in \autoref{ex:intro1}

The code is available at \url{https://mathrepo.mis.mpg.de/ChoppedIdeals/}. It includes a file \texttt{examples.jl} which illustrates some other functionalities, such as computing Hilbert functions, catalecticant matrices and their kernel. Our code performs well, and may be of independent interest for Waring decomposition. On a 16 GB MacBook Pro with an Intel Core i7 processor working at 2.6 GHz, it computes the decomposition of a rank $r = 400$ form of degree $D = 12$ in $n + 1 = 6$ variables with $10$ digits of accuracy within 25 seconds. 

\section*{Future work}

We conclude with some directions for future research. Chopped ideals are relevant for a large class of varieties, besides projective space. For instance, other types of tensor decomposition lead to points in multi-projective space \cite{TelVan:NormalFormAlgorithm}. One can also study ideals of points, and their chopped ideals, in arbitrary toric varieties, rational homogeneous varieties, or other varieties for which it makes sense to consider a multi-graded Hilbert function. This relates to decomposition algorithms and secant varieties as in \cite{BuczBucz:ApolarityBorder,Staf:SchurApolarity,Gal:MultigradedApo}. 

Finally, it is possible to study chopped ideals for positive-dimensional varieties. For instance, there are $7$ sextics passing through $11$ general lines in $\mathbb{P}^3$. These generate a non-saturated chopped ideal, whose saturation is the vanishing ideal of the union of the lines, which has $4$ additional generators in degree 7. For more general classes of varieties, there are several possible choices of chopped ideal to consider, and it would be interesting to explore generalizations of \autoref{conj:chopped_hilbert_function_intro}.

\section*{Acknowledgements}
We would like to thank Edoardo Ballico, Alessandra Bernardi, Luca Chiantini, Liena Colarte-G\'omez, Aldo Conca, Anne Frühbis-Krüger and Alessandro Oneto for helpful conversations and useful pointers to the literature. We thank Jarek Buczy\'nski for his valuable suggestions regarding the proof of \autoref{prop: linear cut}.

\appendix

\section{On the missing sextic}\label{app: missing sextic}

Consider a partition of $Z$ into two disjoint subsets $Z_1,Z_2$, each consisting of $9$ points. Since $h_S(3) = 10$, there are two distinct cubics $g_1,g_2$ such that $I(Z_i)_3 = \langle g_i \rangle_\KK$. We will prove that the sextic $g =g_1g_2$ is not generated by the quintics $f_0,f_1,f_2$; in particular $g \notin ({I_{\langle 5 \rangle}})_6$ and $I(Z) = \langle f_0, f_1,f_2, g \rangle_S$. In order to prove this result, we introduce some geometric tools. 

The first one is an elementary fact about fibers of a branched cover, which is at the foundation of monodromy techniques:
\begin{proposition}\label{prop: monodromy}
Let $f \colon X \to Y$ be a dominant generically finite map of irreducible varieties. Let $Z$ be a closed subvariety of $X$ which intersects general fibers of $f$. Then $Z=X$.
\end{proposition}
\begin{proof}
Since $Z$ intersects general fibers of $f$, we have that $f|_Z$ is dominant too. 
Since $f|_Z$ is dominant and $Z$ is a subvariety of $X$, we have $\dim X \geq \dim Z \geq \dim Y$. On the other hand, since $f$ is dominant and finite, we have $\dim X = \dim Y$. Therefore $\dim Z = \dim X$ and since $X$ is irreducible, we conclude $Z = X$.
\end{proof}

The second result is \autoref{prop: linear cut}, which consists in a generalization of \cite[Lem.~8.1]{BucLan:RanksTensorsAndGeneralization}. In order to prove it, we need the following version of Bertini's Theorem, derived from \cite[Thm.~6.3]{Jou:Bertini}:
\begin{lemma}\label{lemma: bertini}
Let $X \subseteq \bbP^N$ be an irreducible projective variety with singular locus $X_{\sing}$. Let $J$ be a linear series on $X$ with base locus $B \subseteq X$ and let $Y \in J$ be a general element. Then $Y \setminus (X_\sing \cup B)$ is smooth. Moreover, if $\dim J \geq 2$, then $Y \setminus B$ is irreducible.
\end{lemma}
\begin{proof}
Write $m + 1 = \dim J$; the linear series on $X$ defines a regular map $\phi\colon X \setminus B \to \bbP^{m}$ and $Y$ is the (closure of the) the generic fiber of this map.
Consider an affine open cover of $\bbP^{m}$ with the property that every pair of points belongs to at least one affine open subset of the cover. The preimages of the open sets of this cover define a cover of $X \setminus B$ using open quasi-projective varieties. Any pair of points of $X \setminus B$ belongs to at least one quasi-projective variety of this cover.

On each of these open sets the statement is true by \cite[Thm.~6.3]{Jou:Bertini}. Since smoothness can be checked locally, this guarantees that $Y \setminus (X_\sing \cup B)$ is smooth. If $Y \setminus B$ were reducible, consider a quasi-projective open set of the cover which intersects two distinct irreducible components. Then \cite[Thm.~6.3, part~4]{Jou:Bertini} yields a contradiction.
\end{proof}

We now prove that a reduced $0$-dimensional linear section of a linearly non-degenerate variety is itself non-degenerate. This is a higher codimension analog of \cite[Lem.~8.1]{BucLan:RanksTensorsAndGeneralization}.

\begin{proposition}\label{prop: linear cut}
 Let $X \subseteq \bbP^N$ be an irreducible variety of dimension $c$ not contained in a hyperplane. Let $L$ be a linear space with $\codim L = c$. Suppose $X \cap L$ is a set of reduced points. Then $X\cap L$ is not contained in a hyperplane in $L$.
\end{proposition}
\begin{proof}
Let $I(L) = \langle \ell_1 \vvirg \ell_c \rangle$ be the ideal of $L$. Observe that the the points of $X \cap L$ are smooth points of $X$. To see this, let $p \in X \cap L$ and consider the local ring $\calO_{X,p}$; since $X \cap L$ is a set of reduced points, $I(L)_p \subseteq \calO_{X,p}$ coincides with the maximal ideal in $\calO_{X,p}$; in particular (the localizations of) $(\ell_1 \vvirg \ell_c)$ define a regular sequence in $\calO_{X,p}$, showing that $\calO_{X,p}$ is a regular local ring, and equivalently that $p$ is smooth in $X$.

Let $L = L_c \subseteq L_{c-1} \subseteq \cdots \subseteq L_1 \subseteq L_0 = \bbP^N$ be a general flag of linear spaces, with $\codim L_j = j$. For every $j$, define $X_j = X \cap L_j$; in particular $X_c = X \cap L$. For every $j = 0 \vvirg c-1$, we will show that $X_j$ is irreducible and smooth away from the singular locus of $X$. We proceed by induction on $j$. The base case $j = 0$ is straightforward.

For $ j \geq 1$, suppose $X_{j-1}$ is irreducible and smooth away from the singular locus of $X$. Then $I(L)|_{L_{j-1}}$ defines a (non-complete) linear series on $X_j$. Since $L_{j} \subseteq L_{j-1}$ is general among planes containing $L$, \autoref{lemma: bertini} guarantees that $X_j = X_{j-1} \cap L_{j}$ is smooth away from the singularities of $X_{j-1}$ and the base locus of $I(L)$. Moreover, since $j \leq c-1$, $\dim X_{j-1} \geq 2$, hence $X_{j}$ is irreducible except possibly for components supported in the base locus of $I(L)$. The base locus of $I(L)$ is $X \cap L$, which, as shown above, consists of smooth points of $X$. This guarantees that there are no embedded components nor singularities supported on the points of $X \cap L$. We conclude that for every $j = 0 \vvirg c-1$, $X_j$ is irreducible and smooth away from $X_{\rm sing}$. 

An induction argument on $j = 0 \vvirg c-1$, with successive applications of \cite[Lem.~8.1]{BucLan:RanksTensorsAndGeneralization}, shows that $X_{j+1} = X \cap L_{j+1}$ is linearly non-degenerate in $L_{j+1}$. In particular, $X \cap L_c = X_{c-1} \cap L$ is linearly non-degenerate in $L$. This concludes the proof.
\end{proof}

Finally, we will use that the degree of the variety 
\[
\calP _{3,3} \,\coloneqq\, \Set{ g \in \bbP S_6 | g = g_1g_2 \text{ for some $g_1,g_2 \in S_3$} } \,\subseteq\, \bbP S_6
\]
is $\frac{1}{2} \binom{18}{9} = 24310$. This can be computed using elementary intersection theory; see, e.g., \cite[Sec.~2.2.2]{EisHar:3264} for a similar calculation. 
\begin{proposition}
 Let $Z \subseteq \bbP^2$ be a set of $18$ general points. For every bipartition $Z = Z_1 \mathbin{\dot{\cup}} Z_2$ of $Z$ into two sets of $9$ points, one has $g = g_1g_2 \notin I(Z)_{\langle 5 \rangle}$, where $I(Z_i) = \langle g_i \rangle_S$.
\end{proposition}
\begin{proof}
The proof is structured as follows. We first show that there is some partition for which $g = g_1g_2 \in I(Z) \setminus I(Z)_{\langle 5 \rangle}$. Then, we use \autoref{prop: monodromy} to show that the same must hold for all partitions. 

Since $Z$ is general, we have $\dim_\KK I(Z)_6 = 10$, $\dim_\KK (I(Z)_{\langle 5 \rangle})_6 = 9$. Notice $\dim \calP_{3,3} = 9+9 = 18 = \codim_{\mathbb{P}S_6} \bbP I(Z)_6$. Let $W \coloneqq \calP_{3,3} \cap \bbP I(Z)_6 \subseteq \bbP S_6$.

For every $g = g_1g_2 \in W$, observe that $Z_i = Z \cap \{ g_i = 0\}$ defines a bipartition of $Z$ into two subsets of $9$ points. Indeed, $Z = Z_1 \cup Z_2$, and no subset of $10$ points in $Z$ has a cubic equation because of the genericity assumption. On the other hand, any bipartition of $Z = Z_1 \mathbin{\dot{\cup}} Z_2$ into two subsets of $9$ points gives rise to an element $g = g_1g_2 \in W$; by genericity, all these elements are distinct. This shows that $W = \calP_{3,3} \cap \bbP I(Z)_6$ is a set of $\frac{1}{2}\binom{18}{2} = \deg \calP_{3,3}$ points. In particular, $W$ is reduced. 

By \autoref{prop: linear cut}, $W$ is linearly non-degenerate in $\bbP I(Z)_6$, so it is not contained in the hyperplane $\bbP (I(Z)_{\langle 5 \rangle})_6$. This shows that at least one element of $W$ is not contained in the chopped ideal $I(Z)_{\langle 5 \rangle}$.

We now show that $W \subseteq \mathbb{P}I(Z)_6 \setminus \mathbb{P}(I(Z)_{\langle 5 \rangle})_6$. Consider the varieties
\begin{align*}
\calY &\,\coloneqq\, \bar{\Set{ (I(Z)_{\langle 5 \rangle})_6 \in \Gr(9, S_6) | \text{$Z \subseteq \bbP ^2$ is a set of $18$ points in general position} }},\\
\calX &\,\coloneqq\, \bar{ \Set{(Z_1,Z_2,g_1,g_2) \in (\bbP^2)^{9} \times (\bbP^2)^{9} \times \bbP S_3 \times \bbP S_3 | \!\!\begin{array}{c}
Z_1 \cap Z_2 = \emptyset \\ 
Z_1,Z_2, Z_1 \cup Z_2 \text{ in gen'l position} \\
I(Z_i)_3 = \langle g_i \rangle
\end{array}\!\! } }.
\end{align*}
These are projective subvarieties of the Grassmannian $\Gr(9,S_6)$ and of the product $(\bbP^2)^{9} \times (\bbP^2)^{9} \times \bbP S_3 \times \bbP S_3$ respectively.
Define the dominant generically finite rational map 
\begin{equation*}
    \phi \colon \calX \dashto \calY, \qquad (Z_1,Z_2,g_1,g_2) \mapsto (I(Z_1 \cup Z_2)_{\langle 5 \rangle})_6.
\end{equation*}

Let $\calZ$ be the subvariety of $\calX$ defined by 
\[
\calZ \,\coloneqq\, \bar{\set{ (Z_1,Z_2,g_1,g_2) | g_1g_2 \in (I(Z_1 \cup Z_2)_{\langle 5 \rangle})_6 }}.
\]
If $W \cap \bbP ((I(Z)_{\langle 5 \rangle})_6) \neq \emptyset$, then $\calZ$ would intersect the generic fiber of $\phi$. In this case, \autoref{prop: monodromy} guarantees $\calX = \calZ$. This implies that for every $(Z_1,Z_2,g_1,g_2) \in \calZ$, we have $g_1g_2 \in (I(Z_1 \cup Z_2)_{\langle 5 \rangle})_6$. In other words $W \subseteq \bbP (I(Z)_{\langle 5 \rangle})_6$, which contradicts what we saw above. We conclude $W \cap \bbP (I(Z)_{\langle 5 \rangle})_6 = \emptyset$. In other words, every $g = g_1g_2$ satisfies $g \notin I(Z)_{\langle 5 \rangle}$ as desired.
\end{proof}

\bibliographystyle{alphaurl} 
\bibliography{ideals.bib}

\newcommand{\etalchar}[1]{$^{#1}$}
\begin{thebibliography}{BDHM17}

\bibitem[AC22]{AngChi:DescriptionQuartics}
E.~Angelini and L.~Chiantini.
\newblock {On the description of identifiable quartics}.
\newblock {\em Lin. and Mult. Algebra}, page 1–29, 2022.
\newblock \href {https://doi.org/10.1080/03081087.2022.2052004}
  {\path{doi:10.1080/03081087.2022.2052004}}.

\bibitem[Bal87]{ballico1987generators}
E.~Ballico.
\newblock {Generators for the homogeneous ideal of $s$ general points in
  $\mathbb{P}^3$}.
\newblock {\em J. Algebra}, 106(1):46–52, 1987.
\newblock \href {https://doi.org/10.1016/0021-8693(87)90020-2}
  {\path{doi:10.1016/0021-8693(87)90020-2}}.

\bibitem[BB21]{BuczBucz:ApolarityBorder}
W.~Buczyńska and J.~Buczyński.
\newblock {Apolarity, border rank, and multigraded Hilbert scheme}.
\newblock {\em Duke Math. J.}, 170(16):3659–3702, 2021.
\newblock \href {https://doi.org/10.1215/00127094-2021-0048}
  {\path{doi:10.1215/00127094-2021-0048}}.

\bibitem[BCC{\etalchar{+}}18]{bernardi2018hitchhiker}
A.~Bernardi, E.~Carlini, M.~V. Catalisano, A.~Gimigliano, and A.~Oneto.
\newblock The hitchhiker guide to: Secant varieties and tensor decomposition.
\newblock {\em Mathematics}, 6(12):314, 2018.
\newblock \href {https://doi.org/10.3390/math6120314}
  {\path{doi:10.3390/math6120314}}.

\bibitem[BCMT10]{brachat2010symmetric}
J.~Brachat, P.~Comon, B.~Mourrain, and E.~Tsigaridas.
\newblock Symmetric tensor decomposition.
\newblock {\em Lin. Alg. Appl.}, 433(11-12):1851–1872, 2010.
\newblock \href {https://doi.org/10.1016/j.laa.2010.06.046}
  {\path{doi:10.1016/j.laa.2010.06.046}}.

\bibitem[BDHM17]{BerDalHauMou:TensorDecompHomCont}
A.~Bernardi, N.~S. Daleo, J.~D. Hauenstein, and B.~Mourrain.
\newblock {Tensor decomposition and homotopy continuation}.
\newblock {\em Diff. Geom. Appl.}, 55:78–105, 2017.
\newblock \href {https://doi.org/10.1016/j.difgeo.2017.07.009}
  {\path{doi:10.1016/j.difgeo.2017.07.009}}.

\bibitem[BGI11]{BerGimIda:ComputingSymmetricRankSymmetricTensors}
A.~Bernardi, A.~Gimigliano, and M.~Idà.
\newblock {Computing symmetric rank for symmetric tensors}.
\newblock {\em J. Symb. Comput.}, 46(1):34–53, 2011.
\newblock \href {https://doi.org/10.1016/j.jsc.2010.08.001}
  {\path{doi:10.1016/j.jsc.2010.08.001}}.

\bibitem[BL13]{BucLan:RanksTensorsAndGeneralization}
J.~Buczyński and J.~M. Landsberg.
\newblock {Ranks of tensors and a generalization of secant varieties}.
\newblock {\em Lin. Alg. Appl.}, 438(2):668–689, 2013.
\newblock \href {https://doi.org/10.1016/j.laa.2012.05.001}
  {\path{doi:10.1016/j.laa.2012.05.001}}.

\bibitem[BT20]{BerTau:WaringTangCactusDecomp}
A.~Bernardi and D.~Taufer.
\newblock {Waring, tangential and cactus decompositions}.
\newblock {\em J. Math. Pures Appl.}, 143:1–30, 2020.
\newblock \href {https://doi.org/10.1016/j.matpur.2020.07.003}
  {\path{doi:10.1016/j.matpur.2020.07.003}}.

\bibitem[BT21]{bender2021yet}
M.~R. Bender and S.~Telen.
\newblock Yet another eigenvalue algorithm for solving polynomial systems.
\newblock {\em arXiv:2105.08472}, 2021.

\bibitem[Chi19]{Chi:HilbertFunctionsTensorAn}
L.~Chiantini.
\newblock {Hilbert Functions and Tensor Analysis}.
\newblock In {\em {Quantum Physics and Geometry}}, page 125–151. Springer,
  2019.
\newblock \href {https://doi.org/10.1007/978-3-030-06122-7_6}
  {\path{doi:10.1007/978-3-030-06122-7_6}}.

\bibitem[EH16]{EisHar:3264}
D.~Eisenbud and J.~Harris.
\newblock {\em {3264 and {A}ll {T}hat - {A} {S}econd {C}ourse in {A}lgebraic
  {G}eometry}}.
\newblock Cambridge University Press, Cambridge, 2016.

\bibitem[Eis95]{Eis:CommutativeAlgebra}
D.~Eisenbud.
\newblock {\em {Commutative {A}lgebra: with a view toward algebraic geometry}},
  volume 150 of {\em {Graduate Texts in Mathematics}}.
\newblock Springer-Verlag, New York, 1995.

\bibitem[Eis05]{Eisenbud2005:SyzygyBook}
D.~Eisenbud.
\newblock {\em The Geometry of Syzygies: A Second Course in Commutative Algebra
  and Algebraic Geometry}.
\newblock Springer New York, 2005.
\newblock \href {https://doi.org/10.1007/0-387-26456-6_3}
  {\path{doi:10.1007/0-387-26456-6_3}}.

\bibitem[EM99]{emiris1999matrices}
I.~Z. Emiris and B.~Mourrain.
\newblock Matrices in elimination theory.
\newblock {\em J. Symb. Comput.}, 28(1-2):3–44, 1999.
\newblock \href {https://doi.org/10.1006/jsco.1998.0266}
  {\path{doi:10.1006/jsco.1998.0266}}.

\bibitem[EP96]{Eisenbud1996GaleDA}
D.~Eisenbud and S.~Popescu.
\newblock Gale duality and free resolutions of ideals of points.
\newblock {\em Inventiones mathematicae}, 136:419--449, 1996.
\newblock \href {https://doi.org/10.1007/s002220050315}
  {\path{doi:10.1007/s002220050315}}.

\bibitem[EPSW02]{EPSW:ExteriorAlgebraMethods}
D.~Eisenbud, S.~Popescu, F.-O. Schreyer, and C.~Walter.
\newblock {Exterior algebra methods for the minimal resolution conjecture}.
\newblock {\em Duke Math. J.}, 112(2):379–395, 2002.
\newblock \href {https://doi.org/10.1215/S0012-9074-02-11226-5}
  {\path{doi:10.1215/S0012-9074-02-11226-5}}.

\bibitem[Fr{\"o}85]{Fro:IneqHilbertSer}
R.~Fr{\"o}berg.
\newblock {An inequality for Hilbert series of graded algebras}.
\newblock {\em Mathematica Scandinavica}, 56(2):117–144, 1985.
\newblock URL: \url{https://www.jstor.org/stable/24491560}.

\bibitem[Ga{\l}23]{Gal:MultigradedApo}
M.~Ga{\l}\k{a}zka.
\newblock {Multigraded apolarity}.
\newblock {\em Math. Nach.}, 296(1):286–313, 2023.
\newblock \href {https://doi.org/10.1002/mana.202000484}
  {\path{doi:10.1002/mana.202000484}}.

\bibitem[GGR86]{GerGreRob:MonomialIdeals}
A.~V. Geramita, D.~Gregory, and L.~Roberts.
\newblock {Monomial ideals and points in projective space}.
\newblock {\em J. Pure Appl. Algebra}, 40:33–62, 1986.
\newblock \href {https://doi.org/10.1016/0022-4049(86)90029-0}
  {\path{doi:10.1016/0022-4049(86)90029-0}}.

\bibitem[GM84]{geramita1984ideal}
A.~V. Geramita and P.~Maroscia.
\newblock The ideal of forms vanishing at a finite set of points in
  $\mathbb{P}^n$.
\newblock {\em J. Algebra}, 90(2):528–555, 1984.
\newblock \href {https://doi.org/10.1016/0021-8693(84)90188-1}
  {\path{doi:10.1016/0021-8693(84)90188-1}}.

\bibitem[GS]{M2}
D.~R. Grayson and M.~E. Stillman.
\newblock {Macaulay2, a software system for research in algebraic geometry}.
\newblock Available at \url{http://www.math.uiuc.edu/Macaulay2/}.
\newblock (version 1.17.1).

\bibitem[Har92]{Harris1992}
J.~Harris.
\newblock {\em {Algebraic geometry. A first course}}, volume 133 of {\em
  {Graduate Texts in Mathematics}}.
\newblock Springer-Verlag, New York, 1992.

\bibitem[HH99]{HerzogHibi1999Componentwise}
J.~Herzog and T.~Hibi.
\newblock {Componentwise linear ideals}.
\newblock {\em Nagoya Mathematical Journal}, 153(none):141 -- 153, 1999.

\bibitem[HS96]{HirSimp:ResolutionMinimaleGrandNombre}
A.~Hirschowitz and C.~Simpson.
\newblock {La résolution minimale de l’idéal d’un arrangement général
  d’un grand nombre de points dans $\mathbb{P}^n$}.
\newblock {\em Inventiones mathematicae}, 126:467–503, 1996.
\newblock \href {https://doi.org/10.1007/s002220050107}
  {\path{doi:10.1007/s002220050107}}.

\bibitem[IK99]{IarrKan:PowerSumsBook}
A.~Iarrobino and V.~Kanev.
\newblock {\em {Power sums, {G}orenstein algebras, and determinantal loci}},
  volume 1721 of {\em {Lecture Notes in Mathematics}}.
\newblock Springer-Verlag, Berlin, 1999.

\bibitem[Jou83]{Jou:Bertini}
J.-P. Jouanolou.
\newblock {\em {Th\'eor\`emes de Bertini et applications}}, volume~42.
\newblock Springer, 1983.

\bibitem[Lan12]{Lan:TensorBook}
J.~M. Landsberg.
\newblock {\em {Tensors: {G}eometry and {A}pplications}}, volume 128 of {\em
  {Graduate Studies in Mathematics}}.
\newblock American Mathematical Society, Providence, RI, 2012.

\bibitem[LMR23]{LafMasRis:DecompAlgs}
A.~Laface, A.~Massarenti, and R.~Rischter.
\newblock {Decomposition algorithms for tensors and polynomials}.
\newblock {\em SIAM J. Appl. Alg. Geom.}, 7(1):264–290, 2023.
\newblock \href {https://doi.org/10.1137/21M1453712}
  {\path{doi:10.1137/21M1453712}}.

\bibitem[Lor93]{lorenzini1993minimal}
A.~Lorenzini.
\newblock The minimal resolution conjecture.
\newblock {\em J. Algebra}, 156(1):5–35, 1993.
\newblock \href {https://doi.org/10.1006/jabr.1993.1060}
  {\path{doi:10.1006/jabr.1993.1060}}.

\bibitem[Mat87]{Matsu:CommutativeRingTheory}
H.~Matsumura.
\newblock {\em Commutative Ring Theory}.
\newblock Cambridge Studies in Advanced Mathematics. Cambridge University
  Press, 1987.
\newblock \href {https://doi.org/10.1017/CBO9781139171762}
  {\path{doi:10.1017/CBO9781139171762}}.

\bibitem[Mig98]{Mig:IntroLiaison}
J.~Migliore.
\newblock {\em {Introduction to liaison theory and deficiency modules}}, volume
  165.
\newblock Springer Science \& Business Media, 1998.

\bibitem[MN02]{MigNag:LiaisonRelatedTopics}
J.~Migliore and U.~Nagel.
\newblock {Liaison and related topics: Notes from the Torino Workshop/School}.
\newblock {\em arXiv:0205161}, 2002.

\bibitem[Rus16]{Russo2016:ProjGeo}
F.~Russo.
\newblock {\em On the Geometry of Some Special Projective Varieties}.
\newblock Springer International Publishing, 2016.
\newblock \href {https://doi.org/10.1007/978-3-319-26765-4}
  {\path{doi:10.1007/978-3-319-26765-4}}.

\bibitem[Sau85]{Sauer1985:PointsP2}
T.~Sauer.
\newblock {The number of equations defining points in general position.}
\newblock {\em Pacific J. Math.}, 120(1):199 – 213, 1985.
\newblock \href {https://doi.org/10.2140/pjm.1985.120.199}
  {\path{doi:10.2140/pjm.1985.120.199}}.

\bibitem[SSS{\etalchar{+}}]{PointsSource}
M.~Stillman, G.~G. Smith, S.~A. Strømme, D.~Eisenbud, F.~Galetto, and J.~W.
  Skelton.
\newblock {Points: sets of points. Version~3.0}.
\newblock A \texttt{Macaulay2} package available at
  \url{https://github.com/Macaulay2/M2/tree/master/M2/Macaulay2/packages}.

\bibitem[Sta23]{Staf:SchurApolarity}
R.~Staffolani.
\newblock {Schur apolarity}.
\newblock {\em J. Symb. Comp.}, 114:37–73, 2023.
\newblock \href {https://doi.org/10.1016/j.jsc.2022.04.017}
  {\path{doi:10.1016/j.jsc.2022.04.017}}.

\bibitem[Syl52]{Sylv:PrinciplesCalculusForms}
J.~J. Sylvester.
\newblock {On the principles of the calculus of forms}.
\newblock {\em Cambridge and Dublin Math. J.}, 7:52–97, 1852.

\bibitem[Tel20]{telen2020thesis}
S.~Telen.
\newblock {\em Solving Systems of Polynomial Equations}.
\newblock PhD thesis, {KU} {L}euven, {L}euven, {B}elgium, 2020.
\newblock Available at \url{https://simontelen.webnode.page/publications/}.

\bibitem[Tim21]{timme2021numerical}
S.~Timme.
\newblock {\em Numerical nonlinear algebra}.
\newblock PhD thesis, Technische Universität Berlin (Germany), 2021.
\newblock Available at \url{https://sascha.timme.xyz/}.

\bibitem[TV22]{TelVan:NormalFormAlgorithm}
S.~Telen and N.~Vannieuwenhoven.
\newblock A normal form algorithm for tensor rank decomposition.
\newblock {\em ACM Trans. on Math. Soft.}, 48(4):1–35, 2022.
\newblock \href {https://doi.org/10.1145/3555369} {\path{doi:10.1145/3555369}}.

\bibitem[Wal95]{Walt:MinimalFreeResolutionP4}
C.~H. Walter.
\newblock {The minimal free resolution of the homogeneous ideal of $s$ general
  points in $\mathbb{P}^4$}.
\newblock {\em Math. Zeit.}, 219:231–234, 1995.
\newblock \href {https://doi.org/10.1007/BF02572362}
  {\path{doi:10.1007/BF02572362}}.

\end{thebibliography}

\subsection*{Authors' addresses:}

\noindent Fulvio Gesmundo, Universität des Saarlandes; \hfill \href{mailto:fgesmund@math.univ-toulouse.fr}{\tt fgesmund@math.univ-toulouse.fr}\\
\phantom{Fulvio Gesmundo,} {\small (currently) Institut de Mathématiques de Toulouse}

\noindent Leonie Kayser, MPI-MiS Leipzig \hfill \href{mailto:leo.kayser@mis.mpg.de}{\tt leo.kayser@mis.mpg.de}

\noindent Simon Telen, MPI-MiS Leipzig \hfill
\href{mailto:simon.telen@mis.mpg.de}{\tt simon.telen@mis.mpg.de}

\end{document}